\documentclass[11pt,a4paper,thmsb,ukenglish,fleqn]{article}
\usepackage{indentfirst,caption}
\usepackage{varioref,graphicx}
\usepackage{lipsum,amsmath,amssymb,amsthm,esdiff,eucal,mathtools,fixmath,mathrsfs,empheq, xcolor,bbm}
\usepackage{microtype}
\usepackage{geometry}
\usepackage{excludeonly}
\usepackage{epstopdf}
\usepackage[utf8]{inputenc}
\usepackage[sort]{natbib}

\usepackage{color}
\usepackage{a4wide}
\usepackage{appendix}
\usepackage{braket}

\newtheorem{lemma}{Lemma}[section]
\newtheorem{theorem}[lemma]{Theorem}
\newtheorem{proposition}[lemma]{Proposition}

\newtheorem{definition}[lemma]{Definition}
\newtheorem{remark}[lemma]{Remark}
\newtheorem{example}[lemma]{Example}
\newtheorem{assumption}{Assumption}[section]

\allowdisplaybreaks
\newcommand{\filettata}{\mathbb}
\newcommand{\N}{\filettata{N}}
\newcommand{\Hi}{\mathcal H}

\newcommand{\R}{\filettata{R}}
\newcommand{\Prob}{\filettata{P}}

\newcommand{\E}{\mathbb{E}}

\newcommand{\spl}{\begin{split}}
\newcommand{\lit}{\end{split}}

\newcommand{\cadlag}{c\`adl\`ag }

\newcommand{\tot}[1]{^{(#1)}}

\newcommand{\DD}{\Delta^n_i}

\newcommand{\incf}[4]{\frac{\Delta^{#1}_{#2} {#4}\tot{#3}}{\tau\tot{#3}_{#1}}}
\newcommand{\JJ}{\Delta^n_j}
\newcommand{\si}{\sum_{i=1}^n}

\newcommand{\qaz}{{(i-1)\Delta_n}}
\newcommand{\DDint}{\int_{(i-1) \Delta_n}^{i\Delta_n}}

\newcommand{\incratio}[2]{\frac{\Delta^n_{#2} G\tot{#1}}{\tau\tot{#1}_n} }

\newcommand{\stable}{\overset{st.}{\Rightarrow}}
\newcommand{\BSS}{$\mathcal{BSS}$ }
\newcommand{\bss}{$\mathcal{BSS}$}

\newcommand{\Holder}{H\"older }

\newcommand{\leb}{\ensuremath{Leb}}

\allowdisplaybreaks

\renewcommand{\epsilon}{\varepsilon}
\renewcommand{\theta}{\vartheta}
\renewcommand{\phi}{\varphi}

{\left\lbrace\begin{array}{@{}l@{}}}%
{\end{array}\right.}

\DeclarePairedDelimiter{\abs}{\lvert}{\rvert}
\newcommand \norm[1]{\left\lVert #1 \right\rVert}

\DeclarePairedDelimiter\floor{\lfloor}{\rfloor}

\title{A central limit theorem for the realised covariation of a bivariate {B}rownian semistationary process}
\author{Andrea Granelli\thanks{E-mail: \texttt{a.granelli12@imperial.ac.uk}} \and Almut E.~D.~Veraart\thanks{E-mail: \texttt{a.veraart@imperial.ac.uk}}}
\date{
\textit{Department of Mathematics, Imperial College London}\\
\textit{ 180 Queen's Gate, 
 London, SW7 2AZ, 
UK} \\ \ \\
\today}

\begin{document}

\maketitle
\begin{abstract} 

This article presents a weak law of large numbers and a central limit theorem for the  scaled realised covariation of a bivariate Brownian semistationary  process. The novelty of our results lies in the fact that we  derive the suitable asymptotic theory both in a multivariate setting  and outside the classical semimartingale framework.
 The  proofs  rely heavily on recent developments in Malliavin calculus. 

\end{abstract}


\section{Introduction}
Within the realm of   stochastic processes that fail to be a semimartingale, the recent literature has devoted particular attention to the Brownian semistationary ($\mathcal{BSS}$) process, a process that has originally been  used in the context of turbulence modelling in \cite{barndorff2009Brownian}, but has been subsequently employed as a price process in energy markets in \cite{BNBV2013Spot}.
The \BSS process in its most basic form can be written as:
\[
Y_t=\int_{-\infty}^t g(t-s)\sigma_s\,dW_s,
\]
for a deterministic kernel function $g$, a stochastic volatility process $\sigma$ and a Brownian motion $W.$
\cite{pakkanen2011brownian} proved  that $\mathcal{BSS}$ processes have conditional full support  and thus may be used as a price model in financial markets with transaction costs.
Also, \BSS processes can be used in the context of option pricing, through the modelling of \emph{rough volatility} (see \cite{gatheral2014volatility} and \cite{bayer2016pricing}).
In this context,  \cite{bennedsen2015hybrid} present a hybrid simulation scheme used in  Monte Carlo option pricing.

Its spreading use in applications has led to many theoretical questions, some of which have only recently obtained an answer.

Still, the stochastic-analytic properties of the Brownian semistationary process  are not yet completely understood. The univariate case has been studied in detail, and in particular, numerous papers have been published that deal with its asymptotic theory of multipower variation.

The theory of multipower variation for semimartingales was first introduced in \cite{barndorff2004power} and expanded in several subsequent papers (see \cite{barndorfaf2006econometrics}, \cite{barndorff2006central}, \cite{barndorff2006limit}, \cite{kinnebrock2008note}, \cite{jacod2008asymptotic}, \cite{lepingle1976variation}, \cite{vetter2010limit}).
One of the main applications of  multipower variation is the construction of robust estimators that allow to disentangle the impact of the jump risk from the stochastic volatility risk in the price of financial assets.

Outside the semimartingale class a general theory seems to be impossible to achieve and results have to be proved for the particular collection of processes under consideration. For the univariate \BSS process, one can see for example \cite{barndorff2011multipower} with their study of multipower variation through Malliavin calculus and the more recent paper  \cite{barndorff2013limit}  which deals with the multipower variation of  higher order differences of the \BSS process in order to estimate its smoothness.

In the present paper, we define and work with the bivariate Brownian semistationary process.
The introduction of a second dimension greatly increases the complexity, but also allows for novel possibilities in terms of modelling dependence. 
Given the importance in practical applications of the Brownian semistationary process, the first natural result in the multivariate theory must be a limit theorem allowing inference to be performed on the dependence between two components.

In the semimartingale case, inference on the dependence can be performed through the quadratic covariation between two processes. Applying the same ideas to this setting immediately poses the question of whether the quadratic covariation can be successfully defined between two $\mathcal{BSS}$ processes.
There are very few results in the literature concerning quadratic covariation between two non semimartingales. As an example,  \cite{follmercovariation} deal with this problem, but they only consider $[X,F(X)]$, where $X$ is a semimartingale and $F$ is an absolutely continuous function with square integrable derivative. In this case $F(X)$ is not necessarily a semimartingale, while $X$ always is.

We instead propose the study of $[Y\tot 1, Y\tot 2]$, when both $Y\tot 1$ and $Y\tot 2$ are $\mathcal{BSS}$ processes and are not semimartingales.
Hence the aim is to show convergence of an appropriately scaled version of the following realised covariation process:
\begin{equation}
\label{njgfdjlhfds}
\sum_{i=1}^{\floor{nt}} \left(Y\tot1_{\frac in}-Y\tot1_{\frac{i-1}{n}}\right) \left(Y\tot 2_{\frac in}-Y\tot 2_{\frac{i-1}{n}}\right).
\end{equation}

A weak law of large numbers in such a setting has recently been obtained in \cite{GranelliVeraart2017a}. Here, we tackle the arguably more difficult case of deriving a suitable central limit theorem. 
Central limit theorems for processes are results which are usually hard to prove, and techniques to prove them vary from case to case. The most celebrated result of this kind is Donsker's theorem, which states that an appropriately scaled, symmetric random walk converges weakly to Brownian motion (a standard reference is \cite{billingsley2009convergence}).
The high frequency limits of semimartingales are typically processes with a mixed Gaussian distributions, and these central limit theorem results are typically stronger than the standard ones that only state weak convergence in the Skorokhod space, in order for statistical inference to be performed in a  feasible way.
  They instead involve \emph{stable convergence} of processes, which involves proving weak convergence  in an extended sample space, where typically a new Brownian motion lives, which is independent from the original processes. 
We will see that such results can also be obtained in our more general non-semimartingale setting. 

The methods we use in  our proofs rely  heavily
on the powerful \emph{Fourth Moment Theorem} which was proven in \cite{Nualart2005central}. 
Their theory was developed by combining  Stein's method with Malliavin calculus.
The most comprehensive reference on the subject is the monograph \cite{nourdin2012normal}.

The outline of the remainder of this article is as follows. Section \ref{S2} introduces the notation and defines the bivariate Gaussian core and the bivariate Brownian semistationary process. Moreover, we formulate assumptions which ensure that we are outside the semimartingale setting (since the  corresponding theory is well-known in the semimartingale framework).
Section \ref{S3} gives a brief self-contained summary of the key concepts of Malliavin calculus and the celebrated Fourth Moment Theorem needed for proving our results. The main contributions of our article can be found in Sections \ref{S4} and \ref{S5}, where we state the central limit theorems for a suitably scaled version of the realised covariation of a Gaussian core and a Brownian semistationary process, respectively. Section \ref{S6} concludes. The proof of the central limit theorem in the case of the Gaussian core is presented in Section \ref{S7}, and in the case of a Brownian semistationary process in Section \ref{S8}.

\section{The setting}\label{S2}
Throughout this article we denote by $\left(\Omega, \mathscr F,\mathscr F_t, \Prob\right)$  a filtered, complete probability space and by $\mathcal B(\R)$  the class of Borel subsets of $\R$ and we consider a finite time horizon $[0, T]$ for some $T>0$.

We will assume that $\left(\Omega, \mathscr F,\mathscr F_t, \Prob\right)$ supports two  independent $\mathscr F_t$-Brownian measures $W\tot1,  W^{(2)}$ on $\R$, for which we briefly   recall the definition.

\begin{definition}[Brownian measure]
An $\mathscr F_t$-adapted Brownian measure $W\colon \Omega\times\mathcal B(\R)\to \R$ is a Gaussian stochastic measure such that, if $A\in \mathcal B(\R)$ with $\E[(W(A))^2]<\infty$, then $W(A) \sim N(0,\leb(A))$,
where $\leb$ is the Lebesgue measure.
Moreover, if $A\subseteq [t,+\infty)$, then $W(A)$ is independent of $\mathscr F_t$.
\end{definition}

Let us first define the so-called bivariate Gaussian core, which is in fact a bivariate Gaussian moving average process with correlated components. 
\begin{definition}[The Gaussian core]
\label{otioti}
Consider two  Brownian measures $W\tot1$ and $W\tot2$ adapted to $\mathscr F_t$ with 
$dW\tot1_tdW\tot2_t=\rho dt$, for $\rho\in [-1,1]$. Further take 
two nonnegative deterministic functions $g\tot1, g\tot2 \in L^2((0,\infty))$ which are continuous on $\mathbb{R}\setminus\{0\}$. 
Define, for $j\in \{1,2\}$,
\[
G_t\tot j:= \int_{-\infty}^t g\tot j(t-s)  \,dW\tot j_s.
\]
Then the vector process $(\mathbf G_t)_{t\geq 0}=(G^{(1)}_t,G^{(2)}_t)^{\top}_{t\geq 0}$
is called the (bivariate) Gaussian core.
\end{definition}
If we add stochastic volatility to the Gaussian core, then we obtain a bivariate Brownian semistationary (\bss) process defined as follows.
\begin{definition}[Bivariate Brownian semistationary process]
\label{2dimbss}
Consider two  Brownian measures $W\tot1$ and $W\tot2$ adapted to $\mathscr F_t$ with 
$dW\tot1_tdW\tot2_t=\rho dt$, for $\rho\in [-1,1]$. Further take 
two nonnegative deterministic functions $g\tot1, g\tot2 \in L^2((0,\infty))$ which are continuous on $\mathbb{R}\setminus\{0\}$. 
Let further $\sigma \tot1, \sigma\tot2$ be \cadlag, $\mathscr F_t$-adapted stochastic processes and assume that for $j\in\{1,2\}$, and for all $t\in [0,T]$:
$\int_{-\infty}^t g^{(j)2}(t-s) \sigma^{(j)2}_s\,ds<\infty$.
Define, for $j\in \{1,2\}$,
\[
Y_t\tot j:= \int_{-\infty}^t g\tot j(t-s) \sigma_s \tot j \,dW\tot j_s.
\]
Then the vector process $(\mathbf Y_t)_{t\geq 0}=(Y^{(1)}_t,Y^{(2)}_t)^{\top}_{t\geq 0}$
is called a bivariate Brownian semistationary process.
\end{definition}

\subsection{Technical assumptions}
Let us now introduce a few working assumptions. Most of them are standard and already appear in similar forms in the literature, for example in \cite{corcuera2013asymptotic}.

\subsubsection{(Non-) semimartingale conditions}
As mentioned in the introduction, we are exclusively interested in the non-semimartingale setting since the corresponding asymptotic theory for semimartingales is well established in the literature, see e.g.~\cite{protter2005stochastic, barndorff2004econometric}.
It turns out that the (non-) semimartingale property of $G^{(j)}$ or
$Y^{(j)}$ (for $j=1,2$) depends on the properties of the functions $g^{(j)}$.

Let us for a moment suppress the superscripts and write
$G_t=\int_{-\infty}^tg(t-s)dW_s$ for a univariate Gaussian core. 
Consider  the filtration $\left(\mathscr F_t^{W,\infty}\right)_{t\geq0}$ which is the smallest filtration with respect to which $W$ is an adapted Brownian measure and recall the classical result due to 
\cite{knight1992foundations}:
\begin{theorem}[Knight]
The process $(G_t)_{t\geq 0}$ is an $\mathscr F_t^{W,\infty}$-semimartingale if and only if there exists $h \in L^2(\R)$ and $\alpha\in \R$ such that:
$
g(t)=\alpha+\int_0^t h(s)\,ds$.
\end{theorem}

In the case of a univariate Brownian semistationary (\bss) process given by 
\begin{equation}
\label{definition}
Y_t=\int_{-\infty}^t g(t-s)\,\sigma_s dW_s,
\end{equation}
\cite{barndorff2009Brownian} derived the 
following sufficient conditions for a \BSS process $Y$ to be a semimartingale:
\begin{theorem}\label{SCforSM}
Under the assumptions that
(i) $g$ is absolutely continuous and $g'\in L^2((0,\infty))$, 
(ii) $\lim_{x\to 0^+} g(x)=:g(0^+)<\infty$, 
(iii) the process $g'(-\cdot)\sigma_\cdot$ is square integrable,
then $Y_t$ defined as in \eqref{definition} is an $\mathscr F^{W,\infty}_t$-semimartingale.
In this case $Y_t$ admits the decomposition:
\[
Y_t=g(0^+)W_t+\int_0^t\,dl \left[\int_{-\infty}^l g'(l-s)\sigma_s\,dW_s\right].
\]
\end{theorem}

Let us now return to the bivariate case and 
 formulate conditions which ensure that the bivariate  processes ${\bf G}, {\bf Y}$ are \emph{not}  semimartingales. 
This can be achieved by  relaxing the first two assumptions in Theorem \ref{SCforSM} since  both  assumptions are necessary for $G^{(j)}$ to belong to the semimartingale class (see \cite{bassegaussian}) for $j=1,2$.

\begin{assumption}
\label{squareassumm}
For $j\in\{1,2\}$, we assume that $g\tot j\colon\R\to \R^+$ are nonnegative functions and continuous, except possibly at $x=0$. Also, $g\tot j(x)=0$ for $x<0$ and  $g\tot j \in L^2\left((0,+\infty\right))$.
We further ask that $g\tot j$ be differentiable everywhere with derivative $\left(g\tot j\right)'\in L^2((b^{(j)},\infty))$ for some $b^{(j)}>0$ and $\left((g\tot j)'\right)^2$ non-increasing in $[b^{(j)},\infty)$. 
\end{assumption}
In the following we will set $b=\max\{b^{(1)},b^{(2)}\}$, then  $\left(g\tot j\right)'\in L^2((b,\infty))$ and $\left((g\tot j)'\right)^2$ is non-increasing in $[b,\infty)$ for $j=1,2$.

It is important to note that we are not assuming that $\left(g\tot j\right)'\in L^2((0,\infty))$ in order to exclude the semimartingale case. In particular, we must have that, for all $\epsilon >0$, $\sup_{x\in (0,\epsilon)} \left(g\tot j\right)' (x) = \infty.$

%%%%%%%%%%%%%%%%%%%%%%%%%%%%%%%%%%%%
\subsubsection{Technical assumptions for the cross-correlations}

We need some additional technical assumptions to control the terms arising in the covariation between the two components of the bivariate Gaussian core and the bivariate \BSS process.
Such assumptions will be formulated in terms of slowly varying functions, for which we briefly recall the definition, see e.g.~\cite{bingham1989regular}.

\begin{definition}[Slowly and regularly varying function]
A measurable function $L\colon (0,\infty) \to (0,\infty)$ is called \emph{slowly varying at infinity} if, for all $\lambda >0$ we have that 
$\lim_{x\to\infty} \frac {L(\lambda x)}{L(x)} =1$.
A function $g\colon (0,\infty) \to (0,\infty)$ is called \emph{regularly varying at infinity} if, for $x$ large enough, it can be written as:
$g(x)=x^\delta L(x)$,
for a slowly varying function $L$. The parameter $\delta$ is called the \emph{index of regular variation}.
Finally, a  measurable function $L\colon (0,\infty) \to (0,\infty)$ is called \emph{slowly varying at zero} (resp. \emph{regularly varying at zero}) if $x\to L\left(\frac 1x\right)$ is slowly varying (resp. regularly varying) at infinity.
\end{definition}

 For $i,j\in \{1,2\}$, we write $\rho_{i,j}=\rho$ for $i\not =j$ and $\rho_{i,j}=1$ for $i=j$. Also, 
let us introduce the functions mapping $\R^+$ into $\R^+$, with %$i\ne j$ and 
$i,j\in\{1,2\}$:
\begin{equation}
\label{lplotk}
\bar R\tot {i,j} (t) :=\E\left[\left(G\tot j _{t}-G \tot i _0\right)^2\right]
=\norm{g\tot i}_{L^2}^2+\norm{g\tot j}_{L^2}^2-2\E\left[G \tot i_0G\tot j_{t}\right].
\end{equation}
We note that we can write
\begin{align*}
\bar R\tot {i,j} (t)
&=C_{i,j} + 2\rho_{i,j}\int_0^{\infty}(g^{(j)}(x)-g^{(j)}(x+t))g^{(i)}(x)dx,
\end{align*}
 where  $C_{i,j}:=\norm{g\tot i}_{L^2}^2+\norm{g\tot j}_{L^2}^2-2\rho_{i,j}\int_0^{\infty}g^{(i)}(x)g^{(j)}(x)dx$, where in particular $C_{i,i}=0$. 
This enables us to formulate our next assumption.
\begin{assumption} \label{assss}
For all $t\in(0,T)$, there exist
slowly varying functions $L_0\tot {i,j}(t)$ and $L_2\tot{i,j}(t)$ which are  continuous on $(0,\infty)$ such that 
\begin{equation}\label{mfkjfkfid}
\bar R\tot {i,j} (t)= C_{i,j} + \rho_{i,j}
 t^{\delta\tot i+\delta\tot j+1}L_0\tot {i,j}(t), \quad \text{ for  } i,j\in\{1,2\},
 \end{equation}
and
\begin{equation*}
  \frac{1}{2}(\bar R\tot{i,j})''(t)=\rho_{i,j}t^{\delta\tot i+\delta\tot j-1} L_2\tot{i,j}(t), \quad \text{ for  } i,j\in\{1,2\},
\end{equation*}
where $\delta\tot 1, \delta\tot 2\in\left(-\frac 12,\frac12\right)\setminus \{0\}$.

Also, if we denote $\tilde L\tot {i,j}_0(t):={\sqrt{ L\tot{i,i}_0(t) L\tot{j,j}_0(t)  }}$,
we ask that  the functions $L_0\tot {i,j}(t)$ and $L_2\tot{i,j}(t)$ are such that, for all $\lambda >0$, { there exists a $H\tot{i,j}\in \R$ such that:}
\begin{equation}
\label{ewiue}
\lim_{t\to 0+} \frac{L_0\tot {i,j}(\lambda t)}{\tilde L_0\tot{i,j}(t)}= H\tot{i,j}<\infty,
\end{equation}
and that there exists $b\in (0,1)$, such that:
\begin{equation}\label{ass4}
\limsup_{x\to 0^+} \sup_{{y\in(x,x^b)}}\left|\frac{L\tot{i,j}_2(y)}{\tilde L_0^{(i,j)}(x)}\right|<\infty.
\end{equation}
\end{assumption}

In this situation, the restriction $\delta\tot j\in (-\frac12,0)\cup(0,\frac12)$ ensures that the process leaves the semimartingale class.

\begin{remark}
\label{reass}
A  consequence of Assumption \ref{assss} is that:
\[
\sqrt {\bar R\tot {i,i} (t) \bar R\tot {j,j} (t)}=t^{\delta\tot i+\delta \tot j +1}\tilde L\tot{ i,j}_0(t),
\]
where $\tilde L\tot {i,j}_0(t) $ is again a slowly varying function at zero which is continuous on $(0,\infty)$.
\end{remark}

\begin{example} 
In the univariate case, condition \eqref{mfkjfkfid} reads (suppressing superscripts):
\begin{equation}
\label{llkl}
\bar R (t)=  t^{2\delta+1}L_0(t).
\end{equation}
The  so-called  Gamma kernel given by 
$g(x)=e^{-\lambda x} x^{\delta} 1_{\{x>0\}}$,  for $\lambda>0, \delta>-\frac{1}{2}$,
has attracted attention in applications (both to turbulence and finance), see for instance the  
  review paper by \cite{barndorff2016assessing}. In the case when  $\delta \in (-\frac12,0)\cup (0,\frac12],$ $g$ satisfies Assumptions \ref{squareassumm} and condition \eqref{mfkjfkfid}, see   \cite{barndorff2011multipower}. 
\end{example}

\begin{example}
Condition \eqref{ewiue} in Assumption \ref{assss} is satisfied if $\lim_{t\to 0+} L_0\tot{i,j} (t)=M\tot {i,j}<\infty$ and $\lim_{t\to 0+}\sqrt{ L\tot i_0(t) L\tot j_0(t)  } = N\tot{i,j}<\infty$ with $\frac {M\tot {i,j}}{N\tot {i,j}}=H\tot{i,j}$.
In the case when $i=j$, we have $H^{i,j}=1$, so condition  \eqref{ewiue} is satisfied.
\end{example}

As a consequence of Assumption \ref{assss}, we highlight a fact that will be particularly useful for our purposes.
\begin{lemma}
\label{terrific}

Define \begin{align*}
c(x)&:=\int_0^{x} g\tot1(s)g\tot2(s)\,ds+\int_0^\infty\left(g\tot1(s+x)-g\tot1(s)\right)\left(g\tot2(s+x)-g\tot2(s)\right)\,ds.
\end{align*}
If Assumption \ref{assss} holds, then it is possible to  show that:
\begin{equation}\label{pgkfjf}
c(x)= x^{\delta\tot1+\delta\tot2+1} L\tot{1,2}_4(x),
\end{equation}
where $L_4\tot{1,2}$ is a continuous function on $(0,\infty)$ which is slowly varying  at zero, and  $\delta\tot 1, \delta\tot 2\in\left(-\frac 12,\frac12\right)\setminus \{0\}$. Moreover,  there exists a constant $|H|<\infty$ such that 
\begin{align}\label{terrific2}
\lim_{x\to 0+}\frac{L_4^{(1,2)}(x)}{\tilde L_0^{(1,2)}(x)}=H.
\end{align}
More precisely, $H=\frac{1}{2}\left(H\tot{1,2}+H\tot {2,1}\right)$.
\end{lemma}

\begin{example}[Gamma Kernel]
\label{gammagamma}
If the kernel function is the Gamma kernel: 
\[
g \tot i(s) = s^{\delta \tot i } e^{-\lambda \tot i s}  1_{\{s\geq 0\}},
\]
for  $\lambda \tot i>0$, $\delta \tot i \in (-\frac 12, \frac 12) \setminus\{0\}$, and similarly for $g\tot j $, then one can show directly that Lemma \ref{terrific} holds, and give an explicit expression for the constant $H$:
\begin{multline*}
H=\left( -\frac{\Gamma(\delta\tot 1+1)\Gamma(-1-\delta\tot 1 -\delta \tot 2)}{\Gamma(-\delta\tot1)}-\frac{\Gamma(\delta\tot 2+1)\Gamma(-1-\delta\tot 1 -\delta \tot 2)}{\Gamma(-\delta\tot2)}\right)\\2^{1+\delta\tot1+\delta\tot2}\sqrt{\frac{\Gamma(\frac 32+\delta\tot i)\Gamma(\frac 32+\delta\tot j)}{\Gamma(\frac 12-\delta\tot i)\Gamma(\frac 12-\delta\tot j)}}.
\end{multline*}
A proof of this result can be found in Section \ref{proofgammagamma}.
\end{example}

\subsection{Discrete observations and scaling factor}\label{DOSF}
While the stochastic processes we are going to consider are defined in continuous time, we work under the assumption that we only observe them discretely which is the case of practical relevance. Moreover, our asymptotic results  rely on so-called \emph{in-fill asymptotics} where the time interval is fixed, but we sample more and more frequently. This is in contrast to the, in time series more widely used, concept of \emph{long span asymptotics} where the stepsize between observations stays constant, but the number of observations grows, meaning that a bigger and bigger time interval is considered in the asymptotic case.

 Suppose  that we
 sample our processes discretely along successive partitions of $[0,T]$. A partition $\Pi_n$ of $[0,T]$ will be a collection of times $0=t_0<\dots<t_i<t_{i+1}<\dots<t_n=T$, where, for simplicity, we assume that the partition is equally spaced. The mesh of the partition will therefore be $\Delta_n=\frac 1n$ and we have $\lim_{n\to \infty} \Delta_n =0$.

We will use the following notation for (high-frequent) increments of the stochastic processes we are considering: For instance, for the process $G^{(j)}$, we denote its increment by $\DD G^{(j)}:=G_{i\Delta_n}^{(j)}-G_{(i-1)\Delta_n}^{(j)}$, for $j=1,2$. A straightforward computation shows that the increments can be represented as  
\begin{align}
\begin{split}
\DD G^{(j)} &= \int_{-\infty}^\qaz \left(g^{(j)}\left(i\Delta_n-s\right)-g^{(j)}\left(\qaz-s\right)\right)\,dW^{(j)}_s\\
&\qquad + \DDint g^{(j)}(i\Delta_n-s)\,dW^{(j)}_s.
\end{split}
\end{align}
We define the \emph{realised covariation} as
\begin{align*}
\sum_{i=1}^{\lfloor nt \rfloor} \DD G\tot1 \DD G \tot2,
\quad \text{ for } n\geq 1, t\in [0,T]. 
\end{align*}
We know that in the case  when ${\bf G}$ is a semimartingale, then  
\begin{align*}
\sum_{i=1}^{\lfloor nt \rfloor} \DD G\tot1 \DD G \tot2 \overset{\text{u.c.p.}}{\rightarrow} [G^{(1)},G^{(2)}]_t, \quad \text{ as } n\to \infty,
\end{align*}
where the convergence is uniform on compacts in probability (u.c.p.) and the limiting process is the quadratic covariation. However, outside the semimartingale framework, the quadratic covariation does not necessarily exist. 
\cite{GranelliVeraart2017a} recently considered the non-semimartingale case and showed that, under suitable assumptions, the (possibly scaled) realised covariation converges u.c.p.~to an appropriate limit which can be viewed as the correlation between the two non-semimartingale components. In the present work, we would like to go a step further and prove a central limit theorem associated with the scaled realised covariation. 
In order to do so, we need to define the suitable scaling factor. It turns out that the following choice is appropriate. For $j\in \{1, 2\}$, set 
\begin{align}\label{scalingfactor}
\tau _n \tot j :=\sqrt{ \E\left[ \left(\Delta_1^n G\tot j\right)^2\right]}=\sqrt{\int_0^{\infty} \left(g\tot j(s+\Delta_n)-g\tot j(s)\right)^2\,ds + \int_0^{\Delta_n} \left(g\tot j (s)\right)^2\,ds}.
\end{align}

The scaled realised covariation of the Gaussian core is then given by
\begin{align*}
\sum_{i=1}^{\lfloor nt \rfloor} \frac{\DD G\tot1 }{\tau_n^{(1)}}\frac{\DD G \tot2}{\tau_n^{(2)}}.
\end{align*}

Our aim is now to derive a central limit theorem for the suitably centred and scaled realised covariation of the Gaussian core. As soon as we have that result, we will generalise it to the case when the underlying bivariate process is a bivariate Brownian semistationary process and, hence, also accounts for stochastic volatility in each component. 

The key component for proving the two central limit theorems is the so-called Fourth Moment Theorem, see \cite{nourdin2012normal}. Hence we are going to give a very brief self-contained introduction to Malliavin calculus in the next section which will then allow us to formulate the Fourth Moment Theorem. 
%%%%%%%%%%%%%%%%%%%%%%%%%%%%%%%%%%%%%%%%%%%%%%%%%%%%
\section{Pathway to the Fourth Moment Theorem}\label{S3}
The purpose of this section is to illustrate the  background necessary to illustrate the techniques developed by Nualart and Peccati that led them to proving the celebrated \emph{Fourth Moment Theorem}.

We start with an introduction to Malliavin calculus. A good source for this material is Section 2 of \cite{nourdin2012normal}. A good summary of the necessary tools is also presented in \cite{Corcuera2012New}. The standard comprehensive reference for Malliavin calculus is \cite{nualart2006malliavin}.

\subsection{Wiener Chaos decomposition}
We fix a real, separable Hilbert space $\mathcal H$, with its scalar product $\langle \cdot, \cdot \rangle_{\mathcal H}$ and norm $\norm {\cdot}_{\mathcal H} :=\langle \cdot, \cdot \rangle_{\mathcal H}^{\frac 12}$.
We denote by $X=\{X(h) \colon h\in \mathcal H\}$ an \emph{isonormal} Gaussian process over $\mathcal H$ defined on some probability space $(\Omega, \mathscr F, \Prob)$, by which we mean a stochastic process indexed over $\Hi$ such that $\E\left[X(g)X(h)\right]=\langle g,h\rangle_\Hi$,
for every $f,g \in \Hi$. We will assume that $\mathscr F$ is generated by $X$.

The first important result is the granted  existence of an isonormal process:
\begin{proposition}
Given a real, separable Hilbert space $\Hi$, there exists an isonormal process over $\Hi$.
\end{proposition}
\begin{proof}
See \cite{nourdin2012normal}, Theorem 2.1.1.
\end{proof}

We now introduce the fundamental notion of Wiener chaos, which plays a crucial role in our derivation of results.
First, we recall:
\begin{definition}[Hermite polynomials]
Let $p\geq 0$ be an integer. We define the $p$-th Hermite polynomial as $H_0:=1$, for $p=0$, and $H_{p+1}(x):=xH_p(x)-pH_{p-1}(x)$,  for $p>0$.
\end{definition}
\begin{remark}
This is just one of many equivalent definitions for the Hermite polynomials. See \cite{nourdin2012normal}, Definition 1.4.1 and Proposition 1.4.2 for alternative equivalent formulations and characterisations.
\end{remark}

\begin{definition}
For each $n\geq 0$, $\Hi_n$ denotes the closed linear subspace of $L^2(\Omega)$ generated by the random variables $\{H_n\left(X(h)\right)\colon h\in \Hi, \norm{h}_\Hi=1\}$.
The space $\Hi_n$ is called the \emph{$n$-th Wiener chaos of $X$.}
\end{definition}

Wiener chaoses of different order on a Gaussian space are orthogonal as the next proposition shows.
\begin{proposition}
Let $Z,Y\sim \mathscr N(0,1)$ be jointly Gaussian. then, for all $n,m\geq 0$:
\[
\E\left[H_n(Z)H_m(Y)\right]=\begin{cases} n! \left(\E\left[ZY\right]\right)^n, \qquad &\text{if $n=m$}, \\ 0, \qquad &\text{otherwise}.\end{cases}
\]
\end{proposition}
\begin{proof}
See \cite{nourdin2012normal}, Proposition 2.2.1.
\end{proof}
The next theorem states the fundamental fact that the $L^2$-space of random variables can be orthogonally decomposed as a direct sum of Wiener chaoses.
\begin{theorem}[Wiener-It\^o chaos decomposition]
The following decomposition holds:
\[
L^2(\Omega)=\bigoplus_{n=0}^\infty \Hi_n.
\]
So, every variable $F\in L^2(\Omega)$ can be written uniquely as:
\[
F=\E[F]+\sum_{n=1}^\infty F_n,
\]
 where $F_n\in \Hi_n$ and the series converges in $L^2(\Omega)$.
\end{theorem}
\begin{proof}
See Theorem 2.2.4 in \cite{nourdin2012normal}.
\end{proof}
\subsection{Tensor products}
In this section we give a very brief definition of tensor products of Hilbert spaces. The reference that we use here is \cite{reed1908methods}.

Let $\Hi_1$ and $\Hi_2$ be two real Hilbert spaces with inner products $\langle\cdot,\cdot\rangle_{\Hi_1}$ and $\langle\cdot,\cdot\rangle_{\Hi_2}$. For $g\in \Hi_1$ and $h\in \Hi_2$, denote the bilinear form $g\otimes h\colon \Hi_1\times \Hi_2 \to \R$ by:
\[
[g\otimes h](x,y)=\langle x,g\rangle\langle y,h\rangle, \qquad (x,y) \in \Hi_1\times \Hi_2.
\]

Let $\mathcal E$ be the set of all finite linear combinations of such bilinear forms. 
\begin{lemma}
The bilinear form $\ll \cdot, \cdot \gg$ on $\mathcal E$ defined by:
\begin{equation}
\label{zzzzzzz}
\ll g_1\otimes h_1, g_2\otimes h_2\gg:=\langle g_1, g_2\rangle_{\Hi_1}\langle h_1, h_2\rangle_{\Hi_2}
\end{equation}
is symmetric, well defined and positive definite, and thus defines a scalar product on $\mathcal E$.
\end{lemma}
\begin{proof}
See \cite{reed1908methods}.
\end{proof}
The space $\mathcal E$ with the scalar product $\ll\cdot,\cdot\gg$ is obviously not complete. Hence we give the following definition.
\begin{definition}[Tensor product]
The \emph{tensor product} of the Hilbert spaces $\Hi_1$ and $\Hi_2$ is the Hilbert space  $\Hi_1\otimes \Hi_2$ defined to be the completion of $\mathcal E$ under the scalar product  in \eqref{zzzzzzz}.
\end{definition}
Furthermore, we denote by $\Hi^{\otimes n}$ the  $n$-fold tensor product between $\Hi$ and itself.

Symmetric tensors will play an important role in our discussion, and are defined next:
\begin{definition}[Symmetrisation of a tensor product]
If $f\in \Hi^{\otimes n}$ is of the form:
\[
f=h_1 \otimes \dots \otimes h_n,
\]
then the \emph{symmetrisation} of $f$, denoted by $\tilde f$, is defined to be:
\[
\tilde f:=\frac {1}{n!} \sum_{\sigma \in \mathcal S_n} h_{\sigma(1)}\otimes\dots\otimes h_{\sigma (n)},
\]
where the sum is taken over all permutations of \{1,\dots,  n\}.
The closed subspace of $\Hi^{\otimes n}$ generated by the elements of the form $\tilde f$, is called the $n$-fold \emph{symmetric tensor product} of $\Hi$, and is denoted by $\Hi^{\odot n}$.
\end{definition}
A recurrent construction that we will encounter is that of \emph{contracting} a tensor product, defined as follows:
\begin{definition}[Contraction of tensors]
\label{contraction}
Let $g=g_1\otimes\dots\otimes g_n \in \Hi^{\otimes n}$  and $h=h_1\otimes\dots\otimes h_m \in \Hi^{\otimes m}$. For any $0\leq p\leq \min(n,m)$, we define the $p-$th contraction of $g\otimes h$ as the following element of $\Hi^{\otimes m+n-p}$:
\[
 g\otimes_p h := \langle g_1,h_1 \rangle_\Hi\dots \langle g_p,h_p \rangle_\Hi g_{p+1}\otimes\dots\otimes g_n \otimes h_{p+1}\otimes\dots h_m.
\]
Note that, even if $g$ and $h$ are symmetric, their $p$-th contraction is not, in general, a symmetric tensor. We therefore denote by $g\,{\widetilde \otimes}_p h$ its symmetrisation.
\end{definition}

\subsection{The derivative operator}
In this section we define the Malliavin derivative operator. We will need this to define its adjoint operator, the \emph{multiple integral}, that we will use later for the proof of the central limit theorem.

Let $\mathscr S$ denote the set of \emph{smooth} random variables, i.e. of the form:
\begin{equation}
\label{mmmmm}
f\left(X(h_1),\dots, X(h_m)\right),
\end{equation}
where $m\geq 1$, $f$ is a test function, i.e.~$f\in C^\infty$ and $f$ and all of its derivatives have at most polynomial growth and $h_i \in \Hi$, for $i \in\{1,\dots,m\}$.
\begin{lemma}
The space $\mathscr S$ is dense in $L^q(\Omega)$ for every $q\geq 1$.
\end{lemma}
\begin{proof}
See Lemma 2.3.1 in \cite{nourdin2012normal}.
\end{proof}
We need one last technical definition before we can introduce the Malliavin derivative.
\begin{definition}
Given a probability space $(\Omega, \mathscr F, \Prob)$ and a generic real, separable Hilbert space $\Hi$, we denote by $L^q(\Omega, \Hi):=L^q(\Omega, \mathscr F,\Prob; \Hi)$ the class of those $\Hi$-valued random elements $Y$ that are $\mathscr F$-measurable and such that $\norm{Y}^q_\Hi< \infty.$
\end{definition}

We proceed to define  the Malliavin derivative of a smooth variable. 

\begin{definition}[Malliavin Derivative]
Let $F\in \mathscr S$ be given by \eqref{mmmmm}, and $p\geq 1$ an integer. The $p$-th \emph{Malliavin derivative} of $F$ with respect to $X$ is the element of $L^2(\Omega, \Hi^{\odot p})$ defined by:
\[
D^p F:=\sum_{i_1,\dots,i_p=1}^m \frac{\partial ^p}{\partial x_{i_1}\dots \partial x_{i_p}}f\left(X(h_1),\dots X(h_m)\right) h_{i_1}\otimes \dots \otimes h_{i_p}.
\]
\end{definition}

In order for us to define the adjoint of the Malliavin derivative, we need to make sure that the latter operator is at least closable, or else its adjoint could be defined in too small a subset of $L^2(\Omega, \Hi^{\odot p})$. Indeed, recall the following result from functional analysis (we denote by $A^*$ the adjoint of a linear operator $A$):
\begin{proposition}
A linear operator $A \colon D(A)\to H$ is closable if and only if $A^*$ is densely defined.
\end{proposition}

The following theorem establishes  the fundamental fact that the Malliavin operator is indeed closable.
\begin{theorem}
Let $q\in[1,\infty)$, and let $p\geq 1$ be an integer. Then the operator $D^p\colon \mathscr S\subset L^q(\Omega)\to L^q(\Omega, \Hi^{\odot p})$ is closable.
\end{theorem}
\begin{proof}
See Proposition 2.3.4 in \cite{nourdin2012normal}.
\end{proof}
\subsection{The multiple integral}
We are ready to give the formal definition of the multiple divergence operator $\delta^p$.

\begin{definition}
Let $p\geq1 $ be an integer. Denote by $\text{Dom } \delta^p$ the subset of elements $u\in L^2(\Omega, \mathcal H^{\otimes p})$ such that there exists a constant $c$ satisfying:
\begin{equation}\label{adjoint}
\abs{\E\left[\langle D^p F, u\rangle_{\mathcal H^{\otimes p}}\right]}\leq c\sqrt{\E\left[F^2\right]},
\end{equation}
for all $F\in \mathscr S$.
\end{definition}
 Condition \eqref{adjoint} ensures that, for a fixed $u\in \text{Dom }\delta^p$, the linear operator $F\mapsto\E\left[\langle D^p F, u\rangle_{\mathcal H^{\otimes p}}\right]$ is continuous from $\mathscr S$ equipped with the $L^2(\Omega)$ norm into $\R$. Therefore it can be extended to a linear operator from $L^2(\Omega)$ into $\R$. By the Riesz representation theorem, then there exists a unique element in $L^2(\Omega)$, denoted $\delta^p(u)$, such that: $\E\left[\langle D^p F, u\rangle_{\mathcal H^{\otimes p}}\right]=\E\left[F \delta^p(u)\right].$ Thus, we can give the following definition:
\begin{definition}[Multiple divergence]
The multiple divergence operator $\delta^p\colon \text{Dom }\delta^p \subset L^2(\Omega, \Hi^{\otimes p})\to L^2(\Omega)$ is defined to be the adjoint operator of $D^p$. That means that if $u \in \text{Dom }\delta^p$ then $\delta^p(u)$ is defined to be that only element of $L^2(\Omega)$ such that:
\[
\E\left[ F \delta^p(u)\right]=\E\left[\langle u, D^p F\rangle_{\Hi^{\otimes p}}\right],
\]
for all $F\in \mathscr S$.
\end{definition}

Finally, we define the multiple integral operator, which is the object we will need the most in our discussion:
\begin{definition}[Multiple integral]
Let $p\geq 1$ and $f\in \Hi^{\odot p}$. The $p$-th \emph{multiple integral} of $f$ with respect to $X$ is defined to be $I_p(f):=\delta^p(f)$.
\label{multi}
\end{definition}
We further write $I_0:=I$ for the identity in $\R$.

The connection between multiple integrals and the Wiener chaos decomposition is asserted by the following theorem:
\begin{theorem}
Let $f\in \Hi$, with $\norm{ f}_\Hi=1$. Then, for any integer $p\geq 1$, we have:
\[
H_p(\left(X(f)\right)=I_p\left(f^{\otimes p}\right).
\]
As a consequence, the linear operator $I_p$ is an isometry from $\Hi^{\odot p}$ onto the $p-$th Wiener chaos $\Hi_p$ of $X$.
\end{theorem}
\begin{proof}
See Theorem 2.2.7 in \cite{nourdin2012normal}.
\end{proof}
In particular, crucially, the image of a $p-$th multiple integral lies in the $p-$th Wiener chaos of $X$.

We will also make use of the following product formula:
\begin{theorem}[Product formula for multiple integrals]
\label{productmultiple}
Let $p,q\geq 1$. If $f\in \Hi^{\odot p}$ and $g \in \Hi^{\odot q}$, then:
\[
I_p(f)I_q(g)= \sum_{r=0}^{p \wedge q} r! \binom pr \binom qr I_{p+q-2r}(f\, \widetilde \otimes _r g).
\]
\end{theorem}
\begin{proof}
See \cite{nourdin2012normal}, Theorem 2.7.10.
\end{proof}

\subsection{The Fourth Moment Theorem}

With our arsenal of technical tools, we can start to prepare the statement of the fourth moment theorem.
We begin by stating another very remarkable fact. For a vector of $L^2$-variables belonging to a fixed Wiener chaos, joint weak convergence to the Gaussian distribution is equivalent to marginal convergence. More precisely, we have the following theorem:
\begin{theorem}
\label{4thmoment2}
Let $d\geq 2$ and $q_d,\dots ,q_1 \geq 1$ be some fixed integers.
Consider vectors:
\[
\mathbf F_n:=(F_{1,n},\dots,F_{d,n})=(I_{q_1}(f_{1,n}),\dots,I_{q_d}(f_{d,n})), \qquad n\geq 1,
\]
with $f_{i,n}\in \Hi^{\odot q_i}$.  Let $C\in \mathscr M_d(\R)$ be a symmetric, nonnegative definite matrix, and let $\mathbf N \sim  \mathscr N _d(0,C)$. Assume that:
\begin{equation}
\label{fundamental}
\lim_{n\to \infty} \E\left[F_{r,n}F_{s,n}\right]= C(r,s), \qquad 1\leq r,s\leq d.
\end{equation}
Then, as $n\to \infty$ the following two conditions are equivalent:
\begin{itemize}
\item[a)] $\mathbf F_n$ converges in law to $\mathbf N$.
\item[b)] For every $1\leq r \leq d$, $F_{r,n}$ converges in law to $\mathscr{N}(0,C(r,r))$.
\end{itemize} 
\end{theorem}
\begin{proof}
See Theorem 6.2.3 in \cite{nourdin2012normal}.
\end{proof}

We can finally present the statement of the fourth moment theorem, which gives us equivalent conditions for convergence in law when the sequence of variables belongs to a fixed Wiener chaos. 
\begin{theorem}[Fourth moment theorem]
\label{powerfulfriend}
Let $F_n = I_q ( f_n), n \geq 1,$ be a
sequence of random variables belonging to the $q$-th chaos of X, for some fixed
integer $q \geq 2$ (so that $f_n \in \mathcal H^{\odot q}$). Assume, moreover, that $\E[F^2_n] \to \sigma^2 > 0$
as $n\to\infty$. Then, as $n\to\infty$, the following assertions are equivalent:
\begin{enumerate}
\item $F_n \overset{\mathscr L}{\rightarrow}N(0,\sigma^2)$,
\item $\lim_{n\to\infty}\E[F_n^4]= 3\sigma^2$,
\item $\norm{f_n \otimes_r f_n}_{H^{\otimes (2q-2r)}} \to 0$, for all $r=1,\dots,q-1$.
\end{enumerate}
\end{theorem}
\begin{proof}
This is a simplified version of Theorem 5.2.7 in \cite{nourdin2012normal}.
\end{proof}

\section{A central limit theorem for the realised covariation of the Gaussian core}\label{S4}

This section focusses on the Gaussian core ${\bf G}$ as defined in Definition \ref{otioti}; we will use the  notation from Subsection \ref{DOSF} in the following.

Since 
$\mathbf G$ is a Gaussian process, we can apply the Hilbert-space  techniques depicted above, using the Hilbert space of $L^2$-Gaussian variables. To this end, let $\Hi$ be the Hilbert space generated by the random variables given by the scaled increments of the Gaussian core:
\[
\left( \frac{\Delta_i^n G\tot j}{\tau \tot j_n}\right)_{n\geq 1, 1\leq i \leq \floor {nt}, j \in \{1,2\}},
\]
equipped with the scalar product $\langle \cdot, \cdot \rangle_{\Hi}$ induced by $L^2(\Omega, \mathscr F, \Prob)$, i.e., for $X,Y\in \Hi$, we have
$\langle X, Y\rangle_{\Hi}=\E\left[ XY\right]$.

Denoting by $I_d$ the multiple  integral of order $d$, acting on $\Hi^{\odot d}$, with values in $L^2(\Omega),$ (see Definition \ref{multi}),  we can write:
\[
\frac{\DD G\tot1}{\tau\tot1_n}= I_1\left(\frac{\DD G\tot1}{\tau\tot1_n}\right), \qquad  \frac{\DD G\tot2}{\tau\tot2_n}= I_1\left(\frac{\DD G\tot2}{\tau\tot2_n}\right).
\]
Recall the definition of the symmetrisation of the tensor product:
$x \widetilde\otimes y := \frac 12 \left(x\otimes y+ y\otimes x\right)$.
Using the product formula \eqref{productmultiple}, the product of two multiple integrals becomes:
\[
\begin{split}
\frac{\DD G\tot1}{\tau\tot1_n} \frac{\DD G\tot2}{\tau\tot2_n}&=I_1\left(\frac{\DD G\tot1}{\tau\tot1_n}\right)I_1\left(\frac{\DD G\tot2}{\tau\tot2_n}\right)= \sum_{r=0}^1 r! \binom 1r \binom 1r I_{2-2r}\left(\frac{\DD G\tot1}{\tau\tot1_n} \widetilde{\otimes}_r \frac{\DD G\tot2}{\tau\tot2_n}
\right)\\
&=I_2\left(\frac{\DD G\tot1}{\tau\tot1_n} \widetilde{\otimes} \frac{\DD G\tot2}{\tau\tot2_n}\right)+ \E\left[\frac{\DD G\tot1}{\tau\tot1_n} \frac{\DD G\tot2}{\tau\tot2_n}\right].
\end{split}
\]
Rearranging, this yields:
\[
\frac{\DD G\tot1}{\tau\tot1_n} \frac{\DD G\tot2}{\tau\tot2_n}-\E\left[\frac{\DD G\tot1}{\tau\tot1_n} \frac{\DD G\tot2}{\tau\tot2_n}\right]=I_2\left(\frac{\DD G\tot1}{\tau\tot1_n} \widetilde{\otimes} \frac{\DD G\tot2}{\tau\tot2_n}\right).
\]
Let us hence define the function $f\colon L^2(\Omega)\times L^2(\Omega)\to \R$ given by
$f(X,Y)=XY-\E[XY]$, 
and the process:
\[
Z^n_t=\frac{1}{\sqrt n} \sum_{i=1}^{\floor{nt} }f\left(\frac{\DD G\tot1}{\tau\tot1_n}, \frac{\DD G\tot2}{\tau\tot2_n}\right)= \frac{1}{\sqrt n}  \sum_{i=1}^{\floor{nt} } I_2\left( \frac{\DD G\tot1}{\tau\tot1_n} \widetilde{\otimes} \frac{\DD G\tot2}{\tau\tot2_n}\right).
\]

\subsection{A uniform bound for the covariance}
We can now formulate a uniform bound for the covariance term 
 $r^{(n)}_{i,j}(k):=\E\left[ \frac{ \Delta^n_1 G\tot i}{\tau_n\tot i} \frac{\Delta^n_{1+k} G\tot j}{\tau_n\tot j}\right]$, for $i, j \in \{1,2\}$.
\begin{theorem}
\label{ohfinallyyes}
Let $\epsilon>0$, with $\epsilon<1-\delta\tot i -\delta \tot j$, for $i, j\in \{1, 2\}$. Define:
\[
r_{i,j}(k):=(k-1)^{\delta\tot i+ \delta \tot j +\epsilon -1}, \quad \text{if $k> 1$},
\]
and $r_{i,j}(0)=r_{i,j}(1)=1$.
Under Assumptions  \ref{squareassumm} and \ref{assss},  there exists  {a positive constant $C<\infty$ and } a natural number $n_0(\epsilon)$ such that:
\begin{equation}
\label{bound}
\left| r^{(n)}_{i,j}(k)\right| \leq { C } r_{i,j}(k), \quad \text{for $k \geq 0$},
\end{equation}
for all $n\geq n_0(\epsilon)$.
Moreover, define  {  $\rho\tot{i,j}_{\theta}(0) = \rho H$ for $i\ne j$   and $\rho\tot{i,j}_{\theta}(0)=1$ for $i=j$, and for any $i, j \in \{1,2\}$ set}
\begin{align}\label{rhonotation}
\rho\tot{i,j}_{\theta}(k)=\frac 12 \rho_{i,j} H\tot{i,j}\left((k-1)^{\theta} - 2k^{\theta}+(k+1)^{\theta}\right), \quad \text{ for } k\geq 1.
\end{align}
 Then it holds that:
\begin{equation}
\label{thisconverges}
\lim_{n\to\infty}r\tot n_{i,j}(k)= \rho\tot{i,j}_{\delta\tot i+\delta \tot j+1}(k), \quad \text{for all } k\geq 0,\ i, j \in \{1,2\}.
\end{equation}
\end{theorem}

\subsection{Convergence of the finite dimensional distributions of  the Gaussian core}

In order to look at the convergence of the finite-dimensional distributions, let $\{a_k\}, \{b_k\}$ be two increasing sequences of positive  real numbers, with $a_k<b_k<a_{k+1}$, and consider, for any $d\in \N$ the vector:
\[
\begin{pmatrix}
Z^n_{b_1}-Z^n_{a_1}, \dots , Z^n_{b_d}-Z^n_{a_d}  
\end{pmatrix}^\top,
\]
whose generic $k-$th component is:
\[
\frac{1}{\sqrt n} \sum_{i=\floor {n a_k}+1}^{\floor {n b_k}} I_2 \left( \frac{\DD G\tot1}{\tau\tot1_n} \widetilde{\otimes} \frac{\DD G\tot2}{\tau\tot2_n}\right)=I_2\left(\frac{1}{\sqrt n}  \sum_{i=\floor {n a_k}+1}^{\floor {n b_k}}\frac{\DD G\tot1}{\tau\tot1_n} \widetilde{\otimes} \frac{\DD G\tot2}{\tau\tot2_n}\right).
\]

\begin{theorem}[Convergence of the finite dimensional distributions]
\label{finitedisttheo}
Take a Gaussian core as defined in Definition \ref{otioti}. Let Assumptions  \ref{squareassumm} and  \ref{assss}  be satisfied and suppose that 
 $\delta\tot1\in(-\frac 12,\frac 14)\setminus\{0\}, \delta\tot2\in(-\frac 12,\frac 14)\setminus\{0\}$.  
 Consider $f\colon L^2(\Omega)\times L^2(\Omega)\to \R$ given by $f(X,Y)=XY-\E[XY]$, 
and the process:
\[
Z^n_t=\frac{1}{\sqrt n} \sum_{i=1}^{\floor{nt} }f\left(\frac{\DD G\tot1}{\tau\tot1_n} ,\frac{\DD G\tot2}{\tau\tot2_n} \right)= \frac{1}{\sqrt n}  \sum_{i=1}^{\floor{nt} } I_2\left( \frac{\DD G\tot1}{\tau\tot1_n} \widetilde{\otimes} \frac{\DD G\tot2}{\tau\tot2_n}\right).
\]
Let $\{a_k\}, \{b_k\}$ be two increasing sequences of positive  real numbers, with $a_k<b_k<a_{k+1}$, and consider, for any $d\in \N$ the vector:
\[
\mathbf{Z}^n_t:=\begin{pmatrix}
Z^n_{b_1}-Z^n_{a_1}, \dots , Z^n_{b_d}-Z^n_{a_d}  \end{pmatrix}^\top=\begin{pmatrix}F_{1,n}\dots, F_{d,n}
\end{pmatrix}^\top.
\]
Then $\mathbf Z ^n_t\Rightarrow \boldsymbol{ N}\sim \mathscr N_d(\mathbf 0,\mathbf C)$,
where $C_{i,j}=\lim_{n\to \infty} \E\left[F_{i,n}F_{j,n}\right],  1\leq i,j\leq d.$
Finally, the matrix $\mathbf C$ is diagonal, and the general $j$-th diagonal element is equal to $C(1,1)(b_j-a_j)$, with %$C(1,1)$ as in \eqref{mknmknm}.
\begin{multline}\label{mknmknm}
C(1,1):=2\sum_{k=1}^{\infty} \left(  \rho\tot{1,1}_{2\delta\tot 1}(k) \rho\tot{2,2}_{2\delta \tot 2}(k)+ \left(\rho\tot{1,2}_{\delta\tot 1 + \delta \tot 2}(k)\right)^2\right)+(1+\rho^2 { H^2}) < \infty.
\end{multline}
\end{theorem}
In order to compute $C(1,1)$ we remark that the definition of the terms of the form  $\rho_{\theta}^{(i,j)}(k)$ was given in equation  
 \eqref{rhonotation}. 
 
The series in \eqref{mknmknm} converges absolutely, thanks to Theorem \ref{ohfinallyyes}, as it is bounded by:
\[
4\sum_{k=1}^\infty \left(k-1\right)^{2\delta\tot1+2\delta\tot2+2\epsilon -2},
\]
which converges if and only if $2\delta\tot1+2\delta\tot2+2\epsilon -2 < -1 \iff  \delta\tot1+\delta\tot2+\epsilon  < \frac12$, which is implied by our assumption that $\delta\tot1\in(-\frac 12,\frac 14)\setminus\{0\}, \delta\tot2\in(-\frac 12,\frac 14)\setminus\{0\}$.  
\subsection{Tightness of the law of the realised covariation for the Gaussian core}
As customary when proving weak convergence, we also need a tightness result for the law of the realised covariation process. This turns out to be a lot simpler than the convergence of the finite dimensional distributions. 
\begin{theorem}[Tightness]
\label{tightnesstheo}
Let the assumptions as in Theorem \ref{finitedisttheo} hold.
For all $n\in \N$, let $\mathbb P^n$ be the law of the process:
\[
Z^n_\cdot=\frac{1}{\sqrt n} \sum_{i=1}^{\floor{n\cdot} }f\left(\frac{\DD G\tot1}{\tau\tot1_n} ,\frac{\DD G\tot2}{\tau\tot2_n} \right)= \frac{1}{\sqrt n}  \sum_{i=1}^{\floor{n\cdot} } I_2\left( \frac{\DD G\tot1}{\tau\tot1_n} \widetilde{\otimes} \frac{\DD G\tot2}{\tau\tot2_n}\right),
\]
on the Skorokhod space $\mathcal D[0,T]$.
Then, the sequence $\{\mathbb P^n\}_{n\in\N}$ is tight.
\end{theorem}

\subsection{The central limit theorem for the Gaussian core}
With Theorem \ref{finitedisttheo} and \ref{tightnesstheo} at our disposal, it is immediate to prove the fundamental theorem stating weak convergence of the realised covariation of the Gaussian core:
\begin{theorem}[Weak Convergence of the Gaussian Core] \label{weakconGC} With the same  setting and assumptions of Theorem \ref{finitedisttheo},   
we obtain:
\label{weakconv}
\begin{equation}
\left(\frac{1}{\sqrt n}\sum_{i=1}^{\floor {nt}} \left(\incf ni1G \incf ni2G-\E\left[\incf ni1G \incf ni2G\right]\right) \right)_{t\in[0,T]}\Rightarrow \left(\sqrt{\beta} B_t\right)_{t\in[0,T]},
\label{mbkfjdfl}
\end{equation}
where $B_t$ is a Brownian motion independent of the processes $G\tot1$, $G\tot2$, $\beta=C(1,1)$ from \eqref{mknmknm} and the convergence is in the Skorokhod space $\mathcal \mathcal D[0,T]$ equipped with the Skorokhod topology.
\end{theorem}

%%%%%%%%%%%%%%%%%%%%%%%%%%%%%%%%%%%%%%%%%%%%%%%%%%%%%%%%%%%%%%%%%%%%%%%%%%%%%%%%%%%%%%%%%%%%%%%%%%%%%%%%%%%%%%%%%%%%%%%%%%%%%%%%%%%%%%%%%%%%%%%%%%%%%%%%%%%%%%%%%%%%%%%%%%%%%%%%%%%%%%
\section{A central limit theorem for the realised covariation of the Brownian semistationary process}\label{S5}

The weak convergence result for the Gaussian core obtained in the previous section is the cornerstone needed to obtain the general central limit theorem for a Brownian semistationary process ${\bf Y}$, which includes stochastic volatility in each component, recall Definition \ref{2dimbss}.

We will need two additional assumptions:
\begin{assumption}
\label{uetsygvs}
We require that, for $k\in\{1,2\}$, the quantity:
\[
\frac{\sqrt{\E\left[\left(\int_{-\infty}^{(i-1)\Delta_n} \Delta g\tot k \sigma\tot k_s\,dW\tot k_s\right)^2\right]}}{\tau\tot k_n}=\frac{\sqrt{\int_0^\infty \left(g\tot k(s+\Delta_n)-g\tot k(s)\right)^2\E\left[\left(\sigma_{(i-1)\Delta_n-s}\tot k\right)^2\right]\,ds}}{\tau\tot k_n}
\]
is uniformly bounded in $n\in \N $ and $i \in\{1,\dots, n\}$.
\end{assumption}
\begin{example}
Assumption \ref{uetsygvs} is easily satisfied in many cases of interests, for example, if the stochastic volatility processes are second-order stationary.
\end{example}

\begin{assumption}
\label{hsgsgsfs}
The stochastic volatility process $\sigma\tot 1$ (resp. $\sigma\tot 2$)  has $\alpha\tot1$-\Holder  (resp. $\alpha\tot2$) continuous sample paths, for $\alpha\tot1 \in\left(\frac12,1\right)$.
Furthermore, both the kernel functions $g \tot 1$ and $g \tot 2$ satisfy the following property:
For $j \in \{1,2\}$, write:
\[
\pi\tot j_n (A):=\frac{\int_A \left(g\tot j(x+\Delta_n)-g \tot j (x)\right)^2\,ds}{\int_0^\infty \left(g\tot j(x+\Delta_n)-g \tot j (x)\right)^2\,ds}
\]
and note that $\pi_n\tot j$ are probability measures. We ask that there exists a constant $\lambda <-1 $ such that for any $\epsilon_n=O(n^{-\kappa})$, it holds that:
\[
\pi\tot j_n\left((\epsilon_n,\infty)\right)=O\left(n^{\lambda (1-\kappa)}\right).
\]
\end{assumption}
\subsection{Some remarks on stable convergence}
Before stating our central limit theorem result, we need to briefly introduce the notion of \emph{stable convergence}, which is the type of convergence that we will encounter, and which is typically used in inference for semimartingales. 
In this section we take definitions and results from   \cite{aldousmixing}  and from the survey on  uses and properties of stable convergence in \cite{podolskij2010understanding}.

\begin{definition}[Stable convergence]
Let a probability space $(\Omega, \mathscr F, \Prob)$ be fixed.
Suppose the sequence of variables $Y\tot n$ converges weakly to $Y$, denoted by: $$Y\tot n \Rightarrow Y.$$
We say that $Y\tot n$ \emph{converges stably} to $Y$ and write $Y\tot n \overset{st.}{\Rightarrow} Y$ if, for any $\mathscr F-$measurable set $B$, we have:
\[
\lim_{n\to \infty} \Prob\left( \{Y\tot n \leq x\} \cap B\right) = \Prob\left(\{Y\leq x\} \cap B\right),
\]
for a countable, dense set of points $x$.
\end{definition}

It is easy to see that $Y\tot n \overset{st.}{\Rightarrow} Y$, if and only if for any $f$ bounded Borel function, and for any $\mathscr F-$measurable \emph{fixed} variable $Z$:
\[
\lim_{n\to \infty}\E\left[ f\left(Y\tot n\right) Z\right]=\E\left[ f(Y) Z\right].
\]
Yet another characterisation is the following:$$Y\tot n \overset{st.}{\Rightarrow} Y \iff (Y\tot n, Z)\Rightarrow (Y,Z),$$  for any $\mathscr F-$measurable \emph{fixed} variable $Z$.

An obvious consequence of the previous characterisation is the following \emph{continuous mapping theorem} for stable convergence:
\begin{theorem}[Continuous mapping theorem]
\label{continuousmapping}
Suppose that $Y_n\overset{st.}{\Rightarrow} Y$, that $\sigma$ is any fixed $\mathscr F$-measurable random variable and that $g(x,y)$ is a continuous function of two variables. Then:
\[
g(Y_n,\sigma)\overset{st.}{\Rightarrow} g(Y,\sigma).
\]
\end{theorem}

When the limiting variable $Y$ can be taken to be independent of $\mathscr F$, we say that the stable convergence is \emph{mixing}, and we write:
\[
Y\tot n\Rightarrow T \qquad \text{(mixing).}
\]

Finally, there is a useful criterion that can be used to establish  mixing convergence:
\begin{proposition}
\label{dsknskrnavrkn}
Suppose that $Y\tot n \Rightarrow Y$. Then the following are equivalent:
\begin{enumerate}
\item
$Y\tot n \Rightarrow Y \qquad \text{(mixing)},$
\item
For all fixed $k\in \N$ and $B\in \sigma \left(Y\tot 1\dots, Y\tot k\right)$ such that $\Prob (B)>0$,
\[ \lim_{n\to \infty} \Prob\left(Y\tot n\leq x \Big | B\right) = F_Y(x).\]
\end{enumerate}
\end{proposition}
\begin{proof}
See Proposition 2 in \cite{aldousmixing}.
\end{proof}
\subsection{The central limit theorem}
We are now in the position to formulate our key result: the central limit theorem for the suitably centred and scaled realised covariation of a bivariate Brownian semistationary process.

\begin{theorem}[Central limit theorem]
\label{CLT}
Let $\mathscr G$ be the sigma algebra generated by the Gaussian core $\mathbf G$, and let $\sigma\tot1$ and $\sigma\tot2$ be $\mathscr G-$measurable.
For the bivariate \BSS process, provided that Assumptions  \ref{squareassumm}, \ref{assss},   \ref{uetsygvs} and \ref{hsgsgsfs}   are satisfied with 
 $\delta\tot1,\delta\tot2\in(-\frac 12,\frac 14)\setminus\{0\}$, the following $\mathscr G$-stable convergence holds:
\begin{multline}
\left(\frac{1}{\sqrt n} \sum_{i=1}^{\floor {nt}} \incf ni1Y \incf ni2Y-\sqrt n\, \E\left[ \incf n11G \incf n12G\right]\int_0^t \sigma\tot1_s\sigma\tot2_s\,ds\right)_{t\in[0,T]}\\\underset{n\to\infty}{\overset{st.}{\Longrightarrow}} \left(\sqrt \beta \int_0^t \sigma_s \tot1 \sigma_s\tot 2\,dB_s\right)_{t\in[0,T]},\label{CLTFormula}
\end{multline} in the Skorokhod space $ \mathcal D[0,T]$, where $\beta={ C(1,1)}$, see equation  \eqref{mknmknm}.
Also, $B$ is Brownian motion, independent of $\mathscr F$ and defined on  an extension of the filtered probability space $(\Omega,\mathscr F,\mathscr F_t,\Prob)$.
\end{theorem}

We note that the central limit theorem implies a weak law of large numbers, which we present next, cf.~also
\cite{GranelliVeraart2017a}. 
\begin{proposition}\label{LLN}
Assume that the conditions of Theorem \ref{CLT} hold.
Then 
\begin{align*}
\frac{\Delta_n}{c(\Delta_n)} \sum_{i=1}^{\floor{nt}} \Delta^n_iY\tot1 \Delta_i^nY\tot2 \overset{\Prob}{\to} \rho\int_0^t \sigma\tot1_s\sigma\tot2_s\,ds, \qquad \text{ as } n\to \infty.
\end{align*}
\end{proposition}

So Theorem \ref{CLT} implies a weak law of large numbers. 
It is to be stressed though, that the law of large numbers can be formulated in a more general way, modulo some different assumptions on the volatility processes.  We refer to the discussion in \cite{GranelliVeraart2017a} for the details. 
In particular, for the weak law of large numbers to hold,  we do not need the restriction that $\delta\tot 1, \delta\tot 2\in\left(-\frac 12,\frac 14\right)\setminus\{0\}$, but we can have the whole range $\delta\tot 1, \delta\tot 2\in\left(-\frac 12,\frac 12\right)\setminus\{0\}$. 
On the other hand, we remark that the weak law of large numbers formulated in  \cite{GranelliVeraart2017a} required the kernel functions to be decreasing, and we do not have such a restriction for the central limit theorem.

\section{Conclusion}\label{S6}
In this article we have employed techniques that were successfully used in the univariate case for the power, multipower, and bipower variation of the \BSS process and of Gaussian processes, (as appearing in \cite{barndorff2011multipower}, \cite{barndorff2009power}, \cite{barndorff2008bipower}, \cite{corcuera2013asymptotic}) to show a central limit theorem for the realised covariation of the bivariate Gaussian core and the \BSS process. 

This result, apart from being interesting from a purely mathematical point of view, can be viewed as the starting point of the use of multivariate \BSS processes in stochastic modelling. The central limit theorem  unlocks inference on the dependence parameter for the multivariate \BSS process.
There are still parts  of such a multivariate theory that  need to be developed in the future. For instance, one interesting aspect would be to   allow for  the correlation coefficient to be stochastic. Another direction of future research would  include extending our results from the realised covariation to more general functionals, obtaining a fully multidimensional theory of multipower variation of the \BSS process.
 Also, one could investigate whether similar results can be obtained for other forms of volatility modulated Gaussian processes outside the semimartingale setting.

\section{Proofs for the Gaussian core}\label{S7}
\subsection{Proof of Lemma \ref{terrific}}
We start off by proving the very useful Lemma \ref{terrific}.
\begin{proof}[Proof of Lemma \ref{terrific}]
Note that we can express $c(x)$ as follows:
\begin{multline*}
c(x) =\int_0^{x} g\tot1(s)g\tot2(s)\,ds+\int_0^\infty g\tot1(s+x)g\tot2(s+x)\,ds \\- \int_0^\infty g\tot1(s)g\tot2(s+x)\,ds-\int_0^\infty g\tot1(s+x)g\tot2(s)\,ds + \int_0^\infty g\tot 1 (s) g\tot 2 (s)\,ds.
\end{multline*}
After a change of variable, we can write the second integral as: $ \int_x^\infty g\tot1(s)g\tot2(s)\,ds$, and therefore we can simplify the expression as:
\begin{align}\begin{split}
\label{carlotta}
c(x) &=2\int_0^{\infty} g\tot1(s)g\tot2(s)\,ds - \int_0^\infty g\tot 1(s) g\tot 2(s+x)\,ds - \int_0^\infty g\tot 1 (s+x) g \tot 2 (s)\,ds\\
&= \int_0^{\infty} g\tot1(s)(g\tot2(s)-g\tot2(s+x))\,ds
+\int_0^{\infty} g\tot2(s)(g\tot1(s)-g\tot1(s+x))\,ds.
\end{split}\end{align}
 Assumption \ref{assss} implies that 
\begin{align*}
c(x) =  x^{\delta\tot 1+\delta\tot 2+1}\frac{1}{2}\left(L_0\tot {1,2}(x) +  L_0\tot {2,1}(x) \right).
\end{align*}
Note that $L_4\tot {1,2}(x):=\frac{1}{2}\left(L_0\tot {1,2}(x) +  L_0\tot {2,1}(x)\right)$ is itself a slowly varying function  and  the constant $H=\frac{1}{2}\left(H\tot{1,2}+H\tot {2,1}\right)$.
\end{proof}
Let us next provide the details of the computation of $H$ for the case of two Gamma kernels, as discussed in Example \ref{gammagamma}.
\subsection{Proof of Example \ref{gammagamma}}
\label{proofgammagamma}
\begin{proof}[Proof of Example \ref{gammagamma}]
We start with the expression for $c(x)$ given in \eqref{carlotta}:
\[
c(x) =2\int_0^{\infty} g\tot1(s)g\tot2(s)\,ds - \int_0^\infty g\tot 1(s) g\tot 2(s+x)\,ds - \int_0^\infty g\tot 1 (s+x) g \tot 2 (s)\,ds.
\]
If we plug in the explicit epression for the Gamma kernel, we obtain:
\begin{multline}
\label{grett}
c(x) =2\int_0^{\infty} s^{\delta \tot 1+\delta \tot 2} e^{-(\lambda\tot1+\lambda\tot 2)s}\,ds - \int_0^\infty (s^{\delta\tot1}e^{-\lambda\tot1 s} (s+x)^{\delta\tot 2}e^{-\lambda\tot2 (s+x)}\,ds \\- \int_0^\infty (s+x)^{\delta\tot1}e^{-\lambda\tot1 (s+x)} s^{\delta\tot 2}e^{-\lambda\tot2 s}\,ds.
\end{multline}
The first integral can be easily evaluated:
\[
2\int_0^{\infty} g\tot1(s)g\tot2(s)\,ds = 2\frac{ \Gamma (\delta\tot1+\delta\tot2+1)}{(\lambda\tot 1+\lambda\tot2)^{\delta\tot 1+\delta\tot2 +1}}.
\]
The other two integrals can be computed analytically in terms of a power series using formula (12) in \cite{bateman1954tables}[p. 234].
We will use the notation:  $(a)_n=a(a+1)\dots(a+n-1):=\prod_{k=0}^{n-1} (a+k)=\frac{\Gamma(a+n)}{\Gamma(a)}$, with $(a)_0:=1$.

For the first one of the two, for example, the final result is:
\begin{multline}
\label{beastsolved}
K\tot1_1 e^{-\lambda\tot 2 t} x^{\delta \tot 1 +\delta\tot 2 +1}  \sum_{k=0}^\infty \frac{(1+\delta \tot 1)_k}{(\delta \tot 1+ \delta \tot 2+2)_k} \frac {\left((\lambda\tot 1 + \lambda \tot 2)x\right)^k}{k!} \\+ K_2 e^{-\lambda\tot 2 x}\sum_{k=0}^\infty \frac{(\delta \tot 2)_k}{(\delta\tot 1+\delta \tot 2)_k} \frac{((\lambda \tot i+\lambda \tot j)x)^k}{k!},
\end{multline}
for  constants $K_1\tot1,K_2$:
\[
K\tot1_1=\frac{\Gamma(\delta\tot 1+1)\Gamma(-1-\delta\tot 1 -\delta \tot 2)}{\Gamma(-\delta\tot1)}, \qquad
K_2=\frac{ \Gamma (\delta\tot1+\delta\tot2+1)}{(\lambda\tot 1+\lambda\tot2)^{\delta\tot 1+\delta\tot2 +1}}.
\]
Swapping the variables $\delta\tot 1, \delta \tot2$, we obtain the result for the second integral. Summing up, we conclude that \eqref{grett} equals:
\[
\begin{split}
c(x)&=2 K_2
- x^{\delta \tot 1 +\delta\tot 2 +1} \left(K\tot1_1 e^{-\lambda\tot 1 x} f\tot 1(x) +K\tot2_1 e^{-\lambda\tot 2 x} f\tot 2(x) \right)\\ &- K_2\left(e^{-\lambda\tot 1 x }f\tot 3(x)+ e^{-\lambda\tot 2 x}f\tot4(x) \right) ,
\end{split}
\]
where $f\tot 1, f\tot 2$ are power series such that $\lim_{x\to 0} f\tot 1(x) =  \lim_{x\to 0} f\tot 2(x) = 1$, while:
\[
f\tot3 (x)= \sum_{k=0}^\infty \frac{(\delta\tot 1)_k}{(\delta\tot 1+\delta\tot 2)_k}\frac{\left((\lambda\tot1+\lambda\tot 2)x\right)^k}{k!},\qquad f\tot4 (x)= \sum_{k=0}^\infty \frac{(\delta\tot 2)_k}{(\delta\tot 1+\delta\tot 2)_k}\frac{\left((\lambda\tot1+\lambda\tot 2)x\right)^k}{k!}.
\]
Using the Taylor expansion: $e^{-\lambda\tot i x}= 1-\lambda\tot i x +o(x)$, some of the terms simplify to give:
\[
\begin{split}
c(x) &= - x^{\delta \tot 1 +\delta\tot 2 +1} \left(K\tot1_1 e^{-\lambda\tot 1 x} f\tot 1(x) +K\tot2_1 e^{-\lambda\tot 2 x} f\tot 2(x) \right)+ O(x^2)\\
&   =x^{\delta \tot 1 +\delta\tot 2 +1} \left(-K\tot1_1 e^{-\lambda\tot 1 x} f\tot 1(x) -K\tot2_1 e^{-\lambda\tot 2 x} f\tot 2(x)+f\tot 5(x) \right),
\end{split}
\]
and we know that $f\tot 5(x) = O\left( x^{1-\delta\tot1-\delta\tot2 }\right)$.
If we call $L_4\tot{1,2}(x)=-K\tot1_1 e^{-\lambda\tot 1 x} f\tot 1(x) -K\tot2_1 e^{-\lambda\tot 2 x} f\tot 2(x)+f\tot 5(x)$, then $L_4\tot{1,2}(x)$ is continuous and we also have:
\[
\lim_{x\to 0+} L_4\tot{1,2}(x) = -K\tot1_1 - K\tot2_1,
\]
which in particular implies that $L_4\tot{1,2}(x)$ is slowly varying at zero.

We know by \cite{barndorff2011multipower} that:
\[
\lim_{x\to 0+}L_0\tot {i,i}(x)=2^{-1-2\delta\tot i}\frac{\Gamma \left(\frac 12- \delta\tot i\right)}{\Gamma\left(\frac 32+\delta\tot i\right)},
\]
and so:
\[
\lim_{x\to 0+}\tilde L_0\tot{1,2}(x) = K_0 := 2^{-1-\delta\tot1-\delta\tot2}\sqrt{\frac{\Gamma(\frac 12-\delta\tot i)\Gamma(\frac 12-\delta\tot j)}{\Gamma(\frac 32+\delta\tot i)\Gamma(\frac 32+\delta\tot j)}}.
\]
Finally, we can then find an expression for $H$:
\begin{multline*}
H=\frac {-K\tot1_1-K\tot2_1}{K_0}=\left( -\frac{\Gamma(\delta\tot 1+1)\Gamma(-1-\delta\tot 1 -\delta \tot 2)}{\Gamma(-\delta\tot1)}-\frac{\Gamma(\delta\tot 2+1)\Gamma(-1-\delta\tot 1 -\delta \tot 2)}{\Gamma(-\delta\tot2)}\right)\\2^{1+\delta\tot1+\delta\tot2}\sqrt{\frac{\Gamma(\frac 32+\delta\tot i)\Gamma(\frac 32+\delta\tot j)}{\Gamma(\frac 12-\delta\tot i)\Gamma(\frac 12-\delta\tot j)}}.
\end{multline*}
\end{proof}

%%%%%%%%%%%%%%%%%%%%%%%%%%%%%%%%%%%%%%%%%%%%%%%%%%%%%
\subsection{Proof of Theorem \ref{ohfinallyyes}}\label{SectProof_ohfinallyyes}

The uniform bound on the covariances $r^{(n)}_{i,j}(k)$ that we prove on Theorem \ref{ohfinallyyes} is a fundamental analytical result that allows us to sit within the reach of some powerful results of Malliavin calculus. In this section we give the proof of that theorem.
Let us start off with an elementary result.
\begin{lemma}\label{TaylorLemma}
For a $C^2$ function  $u$, and $h>0$:
\[
u(x+h)-2u(x)+u(x-h)=h^2 u''(\zeta),
\]
where $\zeta \in (x-h,x+h)$.
\end{lemma}
\begin{proof}[Proof of Lemma \ref{TaylorLemma}]
Simply write Taylor's formula twice, with Lagrange remainder:
\[
\begin{cases}
u(x+h)=u(x)+hu'(x)+\frac12 h^2u''(\zeta^+), \quad \text{$\zeta^+\in (x,x+h)$},\\
u(x-h)=u(x)-hu'(x)+\frac12 h^2u''(\zeta^-), \quad \text{$\zeta^-\in (x-h,x)$}.
\end{cases}
\]
Adding the two equations:
\[
u(x+h)-2u(x)+u(x-h)= h^2\left(\frac{u''(\zeta^+)+u''(\zeta^-)}{2}\right).
\]
By continuity of $u''$ and the intermediate value theorem: $$\frac{u''(\zeta^+)+u''(\zeta^-)}{2}\in u''\left((\zeta^-,\zeta^+))\subseteq u''\left((x-h,x+h\right)\right),$$ which implies the result.
\end{proof}
We have now  the tools to tackle the proof of Theorem \ref{ohfinallyyes}.
\begin{proof}[Proof of Theorem \ref{ohfinallyyes}]

The objective in the section is to show that we can bound:
\begin{equation}
\label{cuas}
\abs{r^{(n)}_{i,j}(k)}\leq r(k),
\end{equation}
uniformly in $n$, for all choices of $i,j$.
In order to do so, recall the functions mapping $\R^+$ into $\R^+$, with %$i\ne j$ and 
$i,j\in\{1,2\}$:
\[
\bar R\tot {i,j} (t) :=\E\left[\left(G\tot j _{t}-G \tot i _0\right)^2\right].
\]
We need to show that this function is well defined. More generally, note that for the Gaussian core, we have for any $u\in \R$:
\[
\begin{split}
&\E\left[\left(G\tot j_{u+t}-G\tot i_u\right)^2\right]=\E\left[\left(\int_{-\infty}^{u+t} g \tot j (u+t-s) \, dW\tot j_s- \int_{-\infty}^t g\tot i(t-s)\,dW\tot i_s\right)^2\right]\\
&=\int_{-\infty}^{u+t}\left(g\tot j (u+t-s)\right)^2\,ds+\int_{-\infty}^t\left(g\tot i(t-s)\right)^2\,ds-2\int_{-\infty}^t g\tot i (t-s)g\tot j (u+t-s)\rho_{i,j}\,ds\\
&= \int_0^\infty \left(g\tot j(y)\right)^2\,dy+ \int_0^\infty \left(g\tot i (y)\right)^2\,dy-2 \int_0^\infty g\tot i (y) g\tot j (y+t)\rho_{i,j}\,dy\\
&=\norm{g\tot i}_{L^2}^2+\norm{g\tot j}_{L^2}^2-2\E\left[G \tot i_0G\tot j_{t}\right],
\end{split}
\]
 which is indeed a function of $t$ only.
It is straightforward to find the connection between $r^{(n)}_{i,j}(k)$ and $\bar R\tot {i,j}(k)$, when $k\in \N$: 
\begin{align}
\nonumber r^{(n)}_{i,j}(k)&=\E\left[\frac{\Delta_1^n G\tot i}{\tau_n \tot i} \frac{\Delta^n_{1+k} G \tot j}{\tau_n\tot j}\right]=\frac{1}{\tau\tot i_n\tau \tot j _n}\E\left[\left(G\tot i _{\frac 1n}- G\tot i_0\right)\left(G\tot j_{\frac{1+k}{n}}-G\tot j_{\frac kn}\right)\right]\\
&\nonumber=\frac{1}{\tau\tot i_n\tau \tot j _n}\left(\E\left[ G\tot i_{\frac 1n}G\tot j_{\frac{1+k}{n}}\right]-\E\left[G\tot i_{\frac 1n} G \tot j_{\frac kn}\right] - \E\left[G \tot i _0G \tot j_{\frac{1+k}{n}}\right]+ \E\left[G\tot i_0 G\tot j_{\frac kn}\right]\right)\\
&\label{connection}=\frac{1}{\tau\tot i_n\tau \tot j _n}\left(-\bar R \tot{i,j} \left(\frac kn\right) +\frac 12 \bar R \tot{i,j}\left(\frac{k-1}{n}\right)+\frac 12 \bar R \tot {i,j}\left(\frac{k+1}{n}\right)\right)\\
\nonumber &= \frac{1}{2\tau\tot i_n\tau \tot j _n}\left(-2\bar R \tot{i,j} \left(\frac kn\right) + \bar R \tot{i,j}\left(\frac{k-1}{n}\right)+\bar R \tot {i,j}\left(\frac{k+1}{n}\right)\right)\\
\label{connection2}&= \frac{1}{2n^2\tau\tot i_n\tau \tot j _n}\left(\bar R \tot{i,j}\right)''\left(\frac kn+ \frac{\theta^n_k}{n}\right),
\end{align}
for some $\abs{\theta^n_k}<1$, thanks to the  elementary result stated in Lemma \ref{TaylorLemma}.

The connection between $r\tot n_{i,j}$ and $\bar R\tot{i,j}(t)$ was derived in \eqref{connection} and \eqref{connection2}:
\begin{equation}
\label{ppopl}
r\tot n_{i,j}(k)=\frac{-2\bar R\tot{i,j}(\frac kn)+\bar R\tot{i,j}(\frac{k+1}{n})+\bar R\tot{i,j}(\frac{k-1}{n})}{2\sqrt{ \bar R\tot{i,i}(\frac 1n)  \bar R \tot{j,j}(\frac 1n) }}= \frac{1}{2n^2\tau\tot i_n\tau \tot j _n}\left(\bar R \tot{i,j}\right)''\left(\frac kn+ \frac{\theta^n_k}{n}\right),
\end{equation}
as well as: 
\[
\tau_n\tot i =\sqrt{\bar R\tot{i,i}\left(\frac 1n\right)}= \sqrt{\E\left[\left(G\tot i_{\frac 1n}-G\tot i_0\right)^2\right]}=\left(\frac 1n\right)^{
\frac12\left(2\delta\tot i+1\right)}\sqrt{L_0\tot i \left(\frac 1n\right)}.
\]

Let us now show  the uniform bound \eqref{bound} and the limit result for the case when $k\in \N$. 
 For $k\in \mathbb{N}$,  we go back to the second equality in \eqref{ppopl}, and deduce that: {
\[
\abs{r\tot n_{i,j}(k)}=\left|\rho_{i,j}\frac{ \left(k+\theta^n_k\right)^{\delta\tot i+\delta\tot j-1}L_2^{(i,j)}\left(\frac kn + \frac{\theta^n_k }{n}\right)}{\tilde L_0^{(i,j)}(\frac 1n)}\right|\leq \left|\frac{ \left(k+\theta^n_k\right)^{\delta\tot i+\delta\tot j-1}L_2^{(i,j)}\left(\frac kn + \frac{\theta^n_k }{n}\right)}{\tilde L_0^{(i,j)}(\frac 1n)}\right|.
\]
Note that for $k>1$, we deduce from $\theta_k^n\in (-1,1)$ and $\delta\tot i+\delta\tot j-1<-\epsilon <0$ that
\begin{align*}
\left(k+\theta^n_k\right)^{\delta\tot i+\delta\tot j-1}\leq (k-1)^{\delta\tot i+\delta\tot j-1}.
\end{align*}
}

Now, if  $2\leq k< \floor{n^{1-b}},$    then $\frac kn + \frac{\theta^n_k }{n} \in \left(\frac 2n+  \frac{\theta^n_k }{n} ,\frac{\floor{n^{1-b}}-1}{n}
+\frac{\theta^n_k }{n}\right)\subset \left(\frac 1n,\frac{\floor{n^{1-b}}}{n}\right)\subset  \left(\frac 1n,\frac{1}{n^b}\right) $ and hence, 
the bound \eqref{ass4} in Assumption \ref{assss} applies and we obtain that
\[
\frac{L_2^{(i,j)}\left(\frac kn + \frac{\theta^n_k }{n}\right)}{\tilde L_0^{(i,j)}(\frac 1n)}
\]
is bounded close to the origin for $n$ big enough. 

{
If instead $\floor{n^{1-b}} \leq k \leq n$, then, for all $\epsilon >0$, and any $\delta<\epsilon(1-b)$ there exists a constant $C(\delta)>0$ such that 
\begin{align}\begin{split}
\label{jkgfhdfk}
\abs{r\tot n_{i,j}(k+1)}&=\left|\frac{ \left(k+1+\theta^n_{k+1}\right)^{\delta\tot i+\delta\tot j-1}L_2^{(i,j)}\left(\frac{k+1}{n} + \frac{\theta^n_{k+1} }{n}\right)}{\tilde L_0^{(i,j)}(\frac 1n)}\right|\\
&\leq  C(\delta) \frac {k^{\delta\tot i+\delta\tot j-1+\epsilon-\delta}}{n^{\epsilon(1-b)-\delta}} \frac{1}{\tilde L_0^{(i,j)}(\frac 1n)}.
\end{split}
\end{align}
We used the fact that for any $\delta, t > 0$, there exists a constant $C$ depending on $\delta$ (and $t$) only such that  $\abs{L_2(x)}\leq C(\delta)x^{-\delta}$, in a neighborhood $x\in(0,t]$.
}

Observe now that $M^{(i,j)}(n):=\frac{1}{\tilde L_0^{(i,j)}(\frac 1n)}$ is a slowly varying function at $\infty$. Indeed, for any $\lambda>0$:
\[
\lim_{n\to\infty} \frac{M^{(i,j)}(\lambda n)}{M^{(i,j)}(n)}=\lim_{n\to \infty} \frac{\tilde L_0^{(i,j)}(\frac {1}{\lambda n})}{\tilde L_0^{(i,j)}(\frac {1}{ n})}=1.
\]
But since $M^{(i,j)}$ is slowly varying, there exists a constant $\tilde C$ such that,  by Potter's bound:
\[
M(n)\leq \tilde C n^{-\epsilon(1-b)+\delta} \iff  \frac{1}{n^{\epsilon(1-b)-\delta}\tilde L_0^{(i,j)}(\frac 1n)}\leq \tilde C,
\]
that gives us
\[
\abs{r\tot n_{i,j}(k+1)}\leq \tilde C { C(\delta)} k^{\delta\tot i+\delta\tot j-1+\epsilon}.
\]
As $\delta$ is arbitrary, {set $C_1=\tilde C C(\delta)$ for any $\delta>0$. }

Next, let us prove the limit result. To this end, we will use the first equality in \eqref{ppopl} to show the convergence in \eqref{thisconverges}. Using the expression \eqref{mfkjfkfid} from Assumptions \ref{assss}, we get for $k\in \N$:
\begin{align}
\label{klsdlask}
&r\tot n_{i,j}(k)\\&=\rho_{i,j}\frac{-2 k^{\delta\tot i+\delta\tot j+1}L_0^{(i,j)}(\frac kn)+(k-1)^{\delta\tot i+\delta\tot j+1}L_0^{(i,j)}(\frac{k-1}{n})+(k+1)^{\delta\tot i+\delta\tot j+1}L_0^{(i,j)}(\frac{k+1}{n})}{2\tilde L_0^{(i,j)}(\frac 1n)}.\nonumber
\end{align}
Because of \eqref{ewiue}, we get in the limit:
\[
\lim_{n\to \infty} r\tot n_{i,j}(k) =\rho_{i,j}  H^{(i,j)} \frac{\left(-2k ^{\delta\tot i+\delta\tot j+1}+(k-1)^{\delta\tot i+\delta\tot j+1}+(k+1)^{\delta\tot i+\delta\tot j+1} \right)}{2}.
\]

{ Let us now consider the case when $k=0$. We need to show that $\lim_{n\to \infty} r\tot n_{i,j}(0)=\rho H$ for $i\ne j$ and $\lim_{n\to \infty} r\tot n_{i,j}(0)=1$ for $i=j$.
First, suppose that $i=j$. Then  
$r\tot n_{i,j}(0)= 1$, and hence $\lim_{n\to \infty}r\tot n_{i,j}(0)=1$.
Next, assume that $i\ne j$. Then 
\begin{align*}
r\tot n_{i,j}(0)= \rho \frac{\zeta_n}{\xi_n},
\end{align*}
where
\begin{align}\begin{split}\label{zetaxi}
\zeta_n &=\int_0^{\Delta_n} g\tot1(s)g\tot2(s)\,ds+ \int_0^{\infty} \left(g\tot1(s+\Delta_n)-g\tot1(s)\right)\left(g\tot2(s+\Delta_n)-g\tot2(s)\right)\,ds,
\\
\xi_n &=\left[\left(\int_0^{\Delta_n} \left(g\tot1(s)\right)^2\,ds+\int_0^{\infty}\left( g\tot1(s+\Delta_n)-g\tot1(s)\right)^2\,ds\right) \right.
\\
& \qquad \cdot \left. \left(\int_0^{\Delta_n} \left(g\tot2(s)\right)^2\,ds+\int_0^{\infty}\left( g\tot2(s+\Delta_n)-g\tot2(s)\right)^2\,ds\right)\right]^{1/2}.
\end{split}
\end{align}
Using Lemma \ref{terrific}, we get:
\begin{align*}
\zeta_n &=\Delta_n^{\delta^{(1)}+\delta^{(2)}+1}L_4^{(1,2)}(\Delta_n).
\end{align*}
Also, Remark \ref{reass} implies that
\begin{align*}
\xi_n=\Delta_n^{\delta^{(1)}+\delta^{(2)}+1}\sqrt{L_0^{(1)}(\Delta_n)L_0^{(2)}(\Delta_n)}=\Delta_n^{\delta^{(1)}+\delta^{(2)}+1} \tilde L_0^{(1,2)}(\Delta_n).
\end{align*}
Then equation \eqref{terrific2} in 
Lemma \ref{terrific} ensures that
$\lim_{n\to\infty}\zeta_n/\xi_n=H$ and hence 
$\lim_{n\to\infty}r_{i,j}^n(0)=\rho H$ for $i\ne j$.
Finally, we remark that since $r\tot n_{i,j}(0)$ converges, there exists a positive constant $C_2$ such that $|r\tot n_{i,j}(0)|\leq C_2$ for all $n\in \N$. 
So, we can conclude that \eqref{bound} holds with $C=\max\{1,C_1,C_2\}$. }

\end{proof}

\subsection{Limiting covariance}
Our strategy for proving the central limit theorem for the Gaussian core relies on 
Theorem \ref{4thmoment2}, which  gives us the fundamental tool for proving convergence in distribution to a Gaussian variable in this setting. 

In order to be able to apply Theorem \ref{4thmoment2} to prove the central limit theorem later on, we must first 
compute the limiting covariance: i.e.~we need to compute $\lim_{n\to\infty} 
\E\left[I_2(f_{r,n})I_2(f_{s,n})\right]$, where: $$f_{r,n}:=\frac{1}{\sqrt 
n}  \sum_{i=\floor {n a_r}+1}^{\floor {n b_r}}\frac{\DD G\tot1}{\tau\tot1_n} 
\widetilde{\otimes} \frac{\DD G\tot2}{\tau\tot2_n}.$$
We start with the case $r\ne s$:
\begin{multline}
\E\left[I_2(f_{r,n})I_2(f_{s,n})\right]=2\langle f_{r,n},f_{s,n}\rangle_{\Hi^{\otimes 2}}\\=2 \langle \frac{1}{\sqrt n}  \sum_{i=\floor {n a_r}+1}^{\floor {n b_r}}\frac{\DD G\tot1}{\tau\tot1_n} \widetilde{\otimes} \frac{\DD G\tot2}{\tau\tot2_n}, \frac{1}{\sqrt n}  \sum_{j=\floor {n a_s}+1}^{\floor {n b_s}}\frac{\Delta^n_j G\tot1}{\tau\tot1_n} \widetilde{\otimes} \frac{\Delta_j^n G\tot2}{\tau\tot2_n} \rangle_{\Hi^{\otimes 2}}.
\end{multline}
Without loss of generality, we will choose $r=1,s=2,a_1=0,b_1=a_2=1, b_2=2$, obtaining:
\begin{equation}
\label{uywfqjf}
\frac 2n \langle \si \frac{\DD G\tot1}{\tau\tot1_n} \widetilde{\otimes} \frac{\DD G\tot2}{\tau\tot2_n}, \sum_{j=n+1}^{2n} \frac{\Delta_j^n G\tot1}{\tau\tot1_n} \widetilde{\otimes} \frac{\Delta_j^n G\tot2}{\tau\tot2_n}\rangle_{\Hi^{\otimes 2}}.
\end{equation}
Now, let  $k=j-i$.
Also recall the definition $r^{(n)}_{a,b}(k):=\E\left[ \frac{ \Delta^n_1 G\tot a}{\tau_n\tot a} \frac{\Delta^n_{1+k} G\tot b}{\tau_n\tot b}\right]$.
Then, the single scalar product equals:
\begin{align*}
&\quad\langle  \frac{\DD G\tot1}{\tau\tot1_n} \widetilde{\otimes} \frac{\DD G\tot2}{\tau\tot2_n}, \frac{\Delta^n_j G\tot1}{\tau\tot1_n} \widetilde{\otimes} \frac{\Delta^n_j G\tot2}{\tau\tot2_n}\rangle_{\Hi^{\otimes 2}}\\
= &\frac14 \langle \incratio1i \otimes\incratio2i+\incratio2i\otimes\incratio1i,\incratio1j\otimes\incratio2j+\incratio2j\otimes\incratio1j\rangle_{\Hi^{\otimes2}}\\
=&\frac 14  \langle \incratio 1i\otimes\incratio2i,\incratio1j\otimes\incratio2j\rangle_{\Hi^{\otimes2}}+ \frac 14  \langle \incratio 2i\otimes\incratio1i,\incratio1j\otimes\incratio2j\rangle_{\Hi^{\otimes2}}\\
+&\frac 14 \langle \incratio 1i\otimes\incratio2i,\incratio2j\otimes\incratio1j\rangle_{\Hi^{\otimes2}}+\frac 14  \langle \incratio 2i\otimes\incratio1i,\incratio2j\otimes\incratio1j\rangle_{\Hi^{\otimes2}}\\
=&\frac 14 \E\left[ \incratio 1i \incratio1j\right]\E\left[\incratio2i\incratio2j\right]+\frac14 \E\left[ \incratio 2i \incratio1j\right]\E\left[\incratio1i\incratio2j\right]\\
+& \frac 14 \E\left[ \incratio 1i \incratio2j\right]\E\left[\incratio2i\incratio1j\right] + \frac 14 \E\left[ \incratio 2i \incratio2j\right]\E\left[\incratio1i\incratio1j\right]\\
=&\frac 12r\tot n_{1,1}(k)r\tot n_{2,2}(k)+\frac 12 r\tot n_{2,1}(k)r\tot n_{1,2}(k).
\end{align*}
Thus, we have that expression \eqref{uywfqjf} becomes:
\begin{multline}
\label{sbaskbu}
\frac {1}{n}\sum_{k=1}^n k \left(r\tot n_{1,1}(k)r\tot n_{2,2}(k)+ r\tot n_{2,1}(k)r\tot n_{1,2}(k)\right)\\+ \frac {1}{n} \sum_{k=n+1}^{2n-1} \left(2n-k\right)\left(r\tot n_{1,1}(k)r\tot n_{2,2}(k)+ r\tot n_{2,1}(k)r\tot n_{1,2}(k)\right).
\end{multline}
By Cesaro's theorem, if:
\begin{equation}
\label{widbj}
\lim_{k\to\infty} k\left( r\tot n_{1,1}(k)r\tot n_{2,2}(k)+ r\tot n_{2,1}(k)r\tot n_{1,2}(k) \right)=0,
\end{equation}
then the first sum in \eqref{sbaskbu} will converge to zero.
Theorem \ref{ohfinallyyes} %, which will be proven in Subsection \ref{SectProof_ohfinallyyes}, 
gives us:
\begin{equation}
\label{crucial}
 \left| r\tot n_{1,1}(k)r\tot n_{2,2}(k)+ r\tot n_{2,1}(k)r\tot n_{1,2}(k) \right| \leq 2(k-1)^{2(\delta\tot1+\delta\tot2) + 2\epsilon-2}.
\end{equation}
 Hence, we have the limit in \eqref{widbj} provided that $${2(\delta\tot1+\delta\tot2)+2\epsilon-2}<-1\iff(\delta\tot1+\delta\tot2)+\epsilon-1<-\frac12 \iff \epsilon < \frac 12 - (\delta\tot1+\delta\tot2),$$ which, in order for $\epsilon >0$ to hold, implies that we must ask:
 \begin{equation}
 \label{finalcondition}
\delta\tot1+\delta\tot2 < \frac 12.
\end{equation}

Applying Theorem \ref{ohfinallyyes} again shows that  the absolute value of the second sum in \eqref{sbaskbu} can be bounded by: {
\begin{align*}
&\quad \frac 1n \sum_{k=n+1}^{2n-1} (2n-k) 2(k-1)^{2\delta\tot 1+2\delta\tot 2+2\epsilon-2}\\
&=4\sum_{k=n}^{2n-2} k^{{2\delta\tot 1+2\delta\tot 2+2\epsilon-2}}-\frac 2n \sum_{k=n+1}^{2n-1} k (k-1)^{{2\delta\tot 1+2\delta\tot 2+2\epsilon-2}}\\
&=4\sum_{k=n}^{2n-2} k^{{2\delta\tot 1+2\delta\tot 2+2\epsilon-2}}-\frac 2n \sum_{k=n}^{2n-2}  k^{{2\delta\tot 1+2\delta\tot 2+2\epsilon-1}}
-\frac 2n \sum_{k=n}^{2n-2}  k^{{2\delta\tot 1+2\delta\tot 2+2\epsilon-2}}\\
&\leq 4\sum_{k=n}^{2n-2} k^{{2\delta\tot 1+2\delta\tot 2+2\epsilon-2}}+\frac 2n \sum_{k=n}^{2n-2}  k^{{2\delta\tot 1+2\delta \tot 2+2\epsilon -1}}+\frac 2n \sum_{k=n}^{2n-2}  k^{{2\delta\tot 1+2\delta\tot 2+2\epsilon-2}}
 \\
 &\leq 4\sum_{k=n}^{2n-2} k^{{2\delta\tot 1+2\delta\tot 2+2\epsilon-2}}+\frac 4n \sum_{k=n}^{2n-2}  k^{{2\delta\tot 1+2\delta \tot 2+2\epsilon -1}}.
\end{align*}
The first sum goes to zero whenever the summand is summable, thus we get 
{ $\delta^{(1)}+\delta^{(2)}<1$, which is clearly satisfied under condition \eqref{finalcondition}. }
For the second sum, we have in particular that $k<2n \iff \frac 1n < \frac 2k$, so we can write:
\[
\frac 4n \sum_{k=n}^{2n-2}  k^{{2\delta\tot i+2\delta \tot j+2\epsilon -1}}<8\sum_{k=n}^{2n-2} k^{{2\delta\tot i+2\delta\tot j+2\epsilon-2}}.
\]
Condition \eqref{finalcondition} again ensures convergence to zero.

%%%%%%%%%%%%%%%%%%%%%%%%%%%%%%%%%%%%%%%%%%%%%%%%%

\subsection{Limiting variance}
Now we consider the  case when $r=s$ in \eqref{fundamental}, as we have to find the limiting variance. Again, take, by simplicity, $r=s=1,a_1=0,b_1=1$, and this time, $k=\abs{i-j}$:
\begin{align}
\begin{split}
\label{jkasfgjk}
 \E\left[ I_2\left(f_{1,n}\right)I_2\left(f_{1,n}\right)\right]&=2\norm{f_{1,n}}^2_{\Hi^{\otimes2}}\\
&= 2\langle \frac{1}{\sqrt n} \si \frac{\DD G\tot 1}{\tau^{(1)}_n} \tilde\otimes \frac{\DD G\tot 2}{\tau^{(2)}_n}, \frac{1}{\sqrt n} \sum_{j=1}^n \frac{\JJ G\tot 1}{\tau^{(1)}_n} \tilde\otimes \frac{\JJ G\tot 2}{\tau^{(2)}_n}\rangle\\
&=\frac{1}{n}\sum_{i=1}^n  \sum_{j=1}^n \left(r\tot n_{1,1}(k)r\tot n_{2,2}(k)+ r\tot n_{2,1}(k)r\tot n_{1,2}(k)\right).
\end{split}
\end{align}
Now write $r\tot n_{1,1}(k)r\tot n_{2,2}(k)+ r\tot n_{2,1}(k)r\tot n_{1,2}(k)=p_n(\abs{i-j})$ (note that, if $j<i,  r\tot n_{a,b}(j-i)=r\tot n _{b,a}(i-j)$), so that:
\begin{align}\begin{split}
\frac{1}{n}\sum_{i=1}^n  \sum_{j=1}^n p_n(\abs{i-j})&=\frac 2n \si \sum_{j=1}^{i-1} p_n(i-j) + \frac {1}{n} \si p_n(0)= \frac 2n \si \sum_{k=1}^{i-1} p_n(k)+p_n(0)\\
&=\frac 2n \sum_{k=1}^{n-1} \sum_{i=k+1}^n p_n(k)+p_n(0)=\frac 2n \sum_{k=1}^{n-1} (n-k)p_n(k) +  p_n(0)\\
&=2\sum_{k=1}^{n-1}\left(1-\frac kn \right)p_n(k) + p_n(0). \label{iadugvs}
\end{split}\end{align}
Thanks to \eqref{thisconverges}, we see that, for $k\geq1$:
\[
p_n(k)=r\tot n_{1,1}(k)r\tot n_{2,2}(k)+  r\tot n_{2,1}(k)r\tot n_{1,2}(k)\to \rho\tot{1,1}_{2\beta\tot 1}(k) \rho\tot{2,2}_{2\beta \tot 2}(k)+ \left(\rho\tot{1,2}_{\beta\tot 1 + \beta \tot 2}(k)\right)^2,
\]
In the case when $k=0$, we have 
\[
\begin{split}
p_n(0)&=1+\frac{1}{(\tau\tot 1_n \tau\tot 2_n)^2}\left(\E\left[\Delta_1^n G\tot1 \Delta_1^n G\tot 2 \right]\right)^2 
1+\rho^2\left(\frac{\zeta_n}{\xi_n}\right)^2,
\end{split}
\]
where $\zeta_n$ and $\xi_n$ are defined as in \eqref{zetaxi}.
As above, using Assumption \ref{assss} and Lemma \ref{terrific}, 
 ensures that 
$\lim_{n\to\infty}\zeta_n/\xi_n=H^2$ and hence 
$\lim_{n\to\infty}p_n(0)=1+\rho^2 H^2$.

By the bound \eqref{bound} in Theorem \ref{ohfinallyyes} and the bounded convergence theorem, \eqref{iadugvs} converges to \begin{multline}%\label{mknmknm}
C(1,1):=\lim_{n\to \infty}\E\left[ I_2\left(f_{1,n}\right)I_2\left(f_{1,n}\right)\right]\\=2\sum_{k=1}^{\infty} \left(  \rho\tot{1,1}_{2\beta\tot 1}(k) \rho\tot{2,2}_{2\beta \tot 2}(k)+ \left(\rho\tot{1,2}_{\beta\tot 1 + \beta \tot 2}(k)\right)^2\right)+(1+\rho^2 { H^2}) < \infty.\end{multline}

\subsection{Proof of Theorem \ref{finitedisttheo}}

\begin{proof}[Proof of Theorem \ref{finitedisttheo}]

We start from the last statement of the theorem, i.e.~the limiting covariance matrix. The limit: $\lim_{n\to\infty}\E\left[F_{i,n}F_{j,n}\right]$ has been computed in the last few sections, where we picked intervals $[a_k,b_k]$ of length 1 and showed that the matrix is diagonal, with diagonal elements all equal to $C(1,1)$. It is straightforward to change the summation indices in \eqref{jkasfgjk} from $\sum_{i=1}^n$ to $\sum_{i=\floor{na_k}+1}^{\floor {nb_k}}$. Since $\lim_{n\to\infty} \frac {\floor{nb_k}-\floor{na_k}}{n}=\lim_{n\to\infty}\frac{nb_k-\{nb_k\}-na_k+\{na_k\}}{n}=b_k-a_k$, we get the limit as in the statement.

The weak convergence  is now implied by an application of Theorem \ref{4thmoment2}. 
In order to show that condition (b) there is satisfied, we need to check one of the equivalent  conditions provided by Theorem \ref{powerfulfriend}.
Employing condition 3 in our case accounts to verifying that, for $1\leq k\leq d$:
\[
\norm{\left(\frac{1}{\sqrt n}  \sum_{i=\floor {n a_k}+1}^{\floor {n b_k}}\frac{\DD G\tot1}{\tau\tot1_n} \widetilde{\otimes} \frac{\DD G\tot2}{\tau\tot2_n}\right)\otimes_1\left( \frac{1}{\sqrt n}  \sum_{i=\floor {n a_k}+1}^{\floor {n b_k}}\frac{\DD G\tot1}{\tau\tot1_n} \widetilde{\otimes} \frac{\DD G\tot2}{\tau\tot2_n}\right)}_{\Hi^{\otimes 2}}\to0.
\]
Without loss of generality, we look at $d=1$ and assume $a_1=0, b_1=1$:
\[
\frac1n\norm{\sum_{i,j=1}^n  \left(\frac{\DD G\tot1}{\tau\tot1_n} \widetilde{\otimes} \frac{\DD G\tot2}{\tau\tot2_n}\right)\otimes_1  \left(\frac{\Delta^n_j G\tot1}{\tau\tot1_n} \widetilde{\otimes} \frac{\Delta_j^n G\tot2}{\tau\tot2_n}\right)}_{\Hi^{\otimes 2}}.
\]
Let us examine the following: 
\begin{align*}
& \left(\frac{\DD G\tot1}{\tau\tot1_n} \widetilde{\otimes} \frac{\DD G\tot2}{\tau\tot2_n}\right)\otimes_1  \left(\frac{\Delta^n_j G\tot1}{\tau\tot1_n} \widetilde{\otimes} \frac{\Delta_j^n G\tot2}{\tau\tot2_n}\right)\\
&= \left(\frac 12   \frac{\DD G\tot1}{\tau\tot1_n} {\otimes} \frac{\DD G\tot2}{\tau\tot2_n}+\frac 12 \frac{\DD G\tot2}{\tau\tot2_n} {\otimes} \frac{\DD G\tot1}{\tau\tot1_n}\right) \otimes_ 1   \left(\frac 12   \frac{\JJ G\tot1}{\tau\tot1_n} {\otimes} \frac{\JJ G\tot2}{\tau\tot2_n}+\frac 12 \frac{\JJ G\tot2}{\tau\tot2_n} {\otimes} \frac{\JJ G\tot1}{\tau\tot1_n}\right)\\
&=\frac14 \E\left[ \frac{\DD G\tot1}{\tau\tot1_n} \frac{\JJ G\tot1}{\tau\tot1_n}\right]  \frac{\DD G\tot2}{\tau\tot2_n}\otimes  \frac{\JJ G\tot2}{\tau\tot2_n}+\frac14 \E\left[ \frac{\DD G\tot1}{\tau\tot1_n} \frac{\JJ G\tot2}{\tau\tot2_n}\right]  \frac{\DD G\tot2}{\tau\tot2_n}\otimes  \frac{\JJ G\tot1}{\tau\tot1_n}
\\&+\frac14 \E\left[ \frac{\DD G\tot2}{\tau\tot2_n} \frac{\JJ G\tot1}{\tau\tot1_n}\right]  \frac{\DD G\tot1}{\tau\tot1_n}\otimes  \frac{\JJ G\tot2}{\tau\tot2_n}+\frac14 \E\left[ \frac{\DD G\tot2}{\tau\tot2_n} \frac{\JJ G\tot2}{\tau\tot2_n}\right]  \frac{\DD G\tot1}{\tau\tot1_n}\otimes  \frac{\JJ G\tot1}{\tau\tot1_n} \\
&=\frac14\sum_{\substack{\{a,a'\}=\{1,2\}\\ \{b,b'\}=\{1,2\}}} r\tot n_{a,b} (j-i)  \frac{\DD G\tot{a'}}{\tau\tot{a'}_n} \otimes \frac{\JJ G\tot{b'}}{\tau\tot{b'}_n}.
\end{align*}

We hence obtain:
\begin{align}
&\frac{1}{n^2}\norm{\sum_{i,j=1}^n  \left(\frac{\DD G\tot1}{\tau\tot1_n} \widetilde{\otimes} \frac{\DD G\tot2}{\tau\tot2_n}\right)\otimes_1  \left(\frac{\Delta^n_j G\tot1}{\tau\tot1_n} \widetilde{\otimes} \frac{\Delta_j^n G\tot2}{\tau\tot2_n}\right)}_{\Hi^{\otimes 2}}^2 \nonumber\\
 &=\frac{1}{16n^2} \norm{\sum_{i,j=1}^n \left(\sum_{\substack{\{a,a'\}=\{1,2\}\\ \{b,b'\}=\{1,2\}}} r\tot n_{a,b} (j-i)  \frac{\DD G\tot{a'}}{\tau\tot{a'}_n} \otimes \frac{\JJ G\tot{b'}}{\tau\tot{b'}_n}\right) }^2_{\Hi^{\otimes 2}} \nonumber\\
&=\frac {1}{16n^2} \sum_{i,j,i',j'=1}^n \Biggl\langle \sum_{\substack{\{a,a'\}=\{1,2\}\\ \{b,b'\}=\{1,2\}}} r\tot n_{a,b} (j-i)  \frac{\DD G\tot{a'}}{\tau\tot{a'}_n} \otimes \frac{\JJ G\tot{b'}}{\tau\tot{b'}_n},\\&
~~~~~~~~~~~~~~~~~~~~~~ \sum_{\substack{\{\alpha,\alpha'\}=\{1,2\}\\ \{\beta,\beta'\}=\{1,2\}}} r\tot n_{\alpha,\beta} (j'-i')  \frac{\Delta^n_{i'} G\tot{\alpha'}}{\tau\tot{\alpha'}_n} \otimes \frac{\Delta^n_{j'} G\tot{\beta'}}{\tau\tot{\beta'}_n}\Biggr\rangle_{\Hi^{\otimes 2}} \nonumber\\
&= \frac{1}{16n^2}   \sum_{\substack{\{a,a'\}=\{1,2\}\\ \{b,b'\}=\{1,2\}\\\{\alpha,\alpha'\}=\{1,2\}\\ \{\beta,\beta'\}=\{1,2\}} }  \sum_{i,j,i',j'=1}^n r\tot n_{a,b} (j-i) r\tot n_{\alpha,\beta} (j'-i')\left\langle   \frac{\DD G\tot{a'}}{\tau\tot{a'}_n} \otimes \frac{\JJ G\tot{b'}}{\tau\tot{b'}_n}     ,    \frac{\Delta^n_{i'} G\tot{\alpha'}}{\tau\tot{\alpha'}_n} \otimes \frac{\Delta^n_{j'} G\tot{\beta'}}{\tau\tot{\beta'}_n}\right\rangle_{\Hi^{\otimes 2}} \nonumber\\
&= \frac{1}{16n^2}    \sum_{\substack{\{a,a'\}=\{1,2\}\\ \{b,b'\}=\{1,2\}\\\{\alpha,\alpha'\}=\{1,2\}\\ \{\beta,\beta'\}=\{1,2\}} }  \sum_{i,j,i',j'=1}^n r\tot n_{a,b} (j-i) r\tot n_{\alpha,\beta} (j'-i') r\tot n_{a',\alpha'}(i'-i) r\tot n_{b',\beta'}(j'-j) . \label{iyuyuyi}
\end{align}
We need to show that the quantity in \eqref{iyuyuyi} converges to zero. It is sufficient to show that the sum of the absolute values converges to zero.
If we apply H\"older inequality, we get:
\[
\begin{split}
& \frac{1}{16n^2}    \sum_{\substack{\{a,a'\}=\{1,2\}\\ \{b,b'\}=\{1,2\}\\\{\alpha,\alpha'\}=\{1,2\}\\ \{\beta,\beta'\}=\{1,2\}} }  \sum_{i,j,i',j'=1}^n\left| r\tot n_{a,b} (j-i) r\tot n_{\alpha,\beta} (j'-i') r\tot n_{a',\alpha'}(i'-i) r\tot n_{b',\beta'}(j'-j) \right|\\
&\leq \frac {1}{16n^2}   \sum_{\substack{\{a,a'\}=\{1,2\}\\ \{b,b'\}=\{1,2\}\\\{\alpha,\alpha'\}=\{1,2\}\\ \{\beta,\beta'\}=\{1,2\}} }  \sum_{i,j,i',j'=1}^n  \left|  r\tot n_{a,b} (j-i)  r\tot n_{a',\alpha'}(i'-i) \left[\left(r\tot n_{\alpha,\beta} (j'-i')\right)^2+\left(r\tot n_{b',\beta'} (j'-i')\right)^2\right] \right|. \\
\end{split}
\]
So we can split the sum into two components. Let us perform the substitution $$(i,j,i',j')\to(i,j,i',l):=(i,j,i',i'-j').$$ We have:
\begin{align*}
&\frac {1}{16n^2}   \sum_{\substack{\{a,a'\}=\{1,2\}\\ \{b,b'\}=\{1,2\}\\\{\alpha,\alpha'\}=\{1,2\}\\ \{\beta,\beta'\}=\{1,2\}} }  \sum_{i,j,i',j'=1}^n   \left| r\tot n_{a,b} (j-i)  r\tot n_{a',\alpha'}(i'-i) \right|\left(r\tot n_{\alpha,\beta} (j'-i')\right)^2\\
=&\frac {1}{16n^2}   \sum_{\substack{\{a,a'\}=\{1,2\}\\ \{b,b'\}=\{1,2\}\\\{\alpha,\alpha'\}=\{1,2\}\\ \{\beta,\beta'\}=\{1,2\}} }  \sum_{i,j,i'=1}^n  \sum_{\abs l<n} \left| r\tot n_{a,b} (j-i)  r\tot n_{a',\alpha'}(i'-i) \right|\left(r\tot n_{\alpha,\beta} (l)\right)^2\\
&\leq  \frac {1}{16n^2}    \sum_{\substack{\{a,a'\}=\{1,2\}\\ \{b,b'\}=\{1,2\}\\\{\alpha,\alpha'\}=\{1,2\}\\ \{\beta,\beta'\}=\{1,2\}} }  \sum_{l\in \mathbb Z} \left(r\tot n_{\alpha,\beta}(l)\right)^2 \sum_{i,j,i'=1}^n \left| r\tot n_{a,b} (j-i)  r\tot n_{a',\alpha'}(i'-i)  \right|\\
&=  \frac {1}{16n^2}    \sum_{\substack{\{a,a'\}=\{1,2\}\\ \{b,b'\}=\{1,2\}\\\{\alpha,\alpha'\}=\{1,2\}\\ \{\beta,\beta'\}=\{1,2\}} }  \sum_{l\in \mathbb Z} \left(r\tot n_{\alpha,\beta}(l)\right)^2 n 
\sum_{\abs{i}<n}\left| r\tot n _{a,b} (i)\right| \sum_{\abs{j}<n}\left| r\tot n _{a',\alpha'} (j) \right|\\
&= \frac {1}{16n}    \sum_{\substack{\{a,a'\}=\{1,2\}\\ \{b,b'\}=\{1,2\}\\\{\alpha,\alpha'\}=\{1,2\}\\ \{\beta,\beta'\}=\{1,2\}} }  \sum_{l\in \mathbb Z} \left(r\tot n_{\alpha,\beta}(l)\right)^2 
\sum_{\abs{i}<n}\left| r\tot n _{a,b} (i)\right| \sum_{\abs{j}<n}\left| r\tot n _{a',\alpha'} (j) \right|\\
& =\frac{1}{16}   \sum_{\substack{\{a,a'\}=\{1,2\}\\ \{b,b'\}=\{1,2\}\\\{\alpha,\alpha'\}=\{1,2\}\\ \{\beta,\beta'\}=\{1,2\}} }  \sum_{k\in \mathbb Z} \left(r\tot n_{\alpha,\beta}(k)\right)^2 \left[  \frac{1}{\sqrt n}  \sum_{\abs{i}<n} \left|r\tot n _{a,b} (i)\right|  \right] \left[  \frac{1}{\sqrt n}   \sum_{\abs{j}<n} \left|r\tot n _{a',\alpha'} (j)\right|\right].
\end{align*}
Now, fix $\delta>0:$
\[
\frac{1}{\sqrt n} \sum_{\abs i <n }\left|r \tot n_{a,b}(i)\right|= \frac {1}{\sqrt n} \sum_{\abs i \leq \floor{n\delta}}\left| r \tot n_{a,b}(i)\right|+\frac {1}{\sqrt n} \sum_{\floor{n\delta} < \abs{i} < n}\left|r \tot n_{a,b}(i)\right|.
\]
Thanks to  H\"older's inequality, the first term is bounded by: 
\[
\frac {1}{\sqrt n} \sqrt{2 \floor {n\delta} +1} \sqrt{ \sum_{i \in \mathbb Z} \left| r\tot n_{a,b}(i)\right|^2},
\]
and the second one by:
\[
\frac{1}{\sqrt n} \sqrt{ 2(n-\floor {n\delta}-1)} \sum_{\floor{n\delta }<\abs i < n} \left|r \tot n_{a,b}(i)\right|^2.
\]
For a fixed $\delta$,  the second one converges to $0$ as $n$ tends to infinity. The first one is bounded by $K \sqrt{\delta}$ (for a positive constant $K<\infty$), thus letting $\delta\to 0$ we have:
\[
\lim_{n\to \infty} \frac{1}{\sqrt n} \sum_{\abs i <n }\left|r \tot n_{a,b}(i)\right|=\lim_{\delta\to 0}\lim_{n\to \infty} \frac{1}{\sqrt n} \sum_{\abs i <n }\left|r \tot n_{a,b}(i)\right|=0,
\]
provided that the following series
\[
\sum_{i\in \mathbb Z} \left(r\tot n_{\alpha,\beta}(i)\right)^2, \qquad \sum_{i\in \mathbb Z} \left(r\tot n_{a,b}(i)\right)^2, \qquad\sum_{i\in \mathbb Z} \left(r\tot n_{a',b'}(i)\right)^2,
\]
converge.
In a completely analogous way we can show that the second component of the original sum converges to zero, provided that also $ \sum_{i\in \mathbb Z} \left(r\tot n_{b',\beta'}(i)\right)^2$ is finite.
But since from Theorem \ref{ohfinallyyes} we have:
\[
\left| r^{(n)}_{i,j}(k)\right|^2 \leq C(k-1)^{2\delta\tot i+ 2\delta \tot j +2\epsilon -2}, \quad \text{for $k \geq 2$},
\]
it is sufficient to ask that $\delta\tot i +\delta \tot j < \frac 12$ for all possible choices of $i,j$. Explicitly, it is sufficient to ask that:
$\delta \tot 1 <\frac 14 $ and $\delta \tot2 < \frac 14$.

The statement of the theorem is proved.

\end{proof}

\subsection{Proof of Theorem \ref{tightnesstheo}}
\begin{proof}[Proof of Theorem \ref{tightnesstheo}]
\[
\begin{split}
&\E\left[(Z^n_t-Z^n_s)^2\right]=\E\left[(Z^n_{t-s})^2\right]= \E\left[\frac 1n \left(I_2\left(\sum_{i=1}^{\floor{nt}-\floor{ns}}\frac{\Delta^n_i G \tot 1}{\tau\tot 1_n} \otimes \frac{\Delta^n_i G \tot 2}{\tau\tot 2_n} \right)\right)^2\right]\\
&=\frac 1n \langle \sum_{i=1}^{\floor{nt}-\floor{ns}} \frac{\Delta^n_i G \tot 1}{\tau\tot 1_n} \otimes \frac{\Delta^n_i G \tot 2}{\tau\tot 2_n}  ,\sum_{j=1}^{\floor{nt}-\floor{ns}}\frac{\Delta^n_j G \tot 1}{\tau\tot 1_n} \otimes\frac{\Delta^n_j G \tot 2}{\tau\tot 2_n} \rangle_{\Hi^{\otimes 2}}\\
&=\frac {1}{2n}\sum_{i=1}^{\floor{nt}-\floor{ns} }\sum_{j=1}^{\floor{nt}-\floor{ns} }\left( r\tot n_{1,1}(|i-j|)r\tot n_{2,2}(|i-j|)+\frac 12 r\tot n_{2,1}(|i-j|)r\tot n_{1,2}(|i-j|) \right).
\end{split}
\]
Multiplying and dividing by $\floor{nt}-\floor{ns}$ yields:
\[
\frac{\floor{nt}-\floor{ns}}{n} \left(   \frac{1}{\floor{nt}-\floor{ns}}\sum_{k=1}^{\floor{nt}-\floor{ns}-1} \left(1-\frac kn\right) p_n(k) +p_n(0) \right),
\]
thanks to the same arguments as in equation \eqref{iadugvs}. We now know that the quantity in brackets is convergent, hence bounded. Tightness now follows as in the proof of Theorem 7 in \cite{Corcuera2012New}, invoking Theorem 13.5 in \cite{billingsley2009convergence}.
\end{proof}

\subsection{Proof of Theorem \ref{weakconGC}}
\begin{proof}[Proof of Theorem \ref{weakconGC}]
The fact that the finite dimensional distributions of the realised covariation converge to those of Brownian motion is the content of Theorem \ref{finitedisttheo}: the limiting finite dimensional distributions we had there coincide with those on the right hand side of \eqref{mbkfjdfl}.
The fact that the limiting Brownian motion $B_t$ is independent of $G\tot1$ and $G\tot2$ follows from the fact that 
\begin{align*}
\begin{pmatrix}G\tot1_{b_k}-G\tot1_{a_k}, G\tot2_{b_k}-G\tot2_{a_k},\frac{1}{\sqrt n}\sum_{i=\floor{na_k}+1}^{\floor {nb_k}} \left(\incf ni1G \incf ni2G-\E\left[\incf ni1G \incf ni2G\right]\right) \end{pmatrix}_{n\in \N}
\end{align*}
 converges to a multivariate Gaussian, and, for all $n\in\N$ the third component is orthogonal to the first two, as it belongs to a different Wiener chaos.
Given the  tightness result in Theorem \ref{tightnesstheo}, an application of Theorem 13.1 in \cite{billingsley2009convergence} allows to conclude.
\end{proof}

%%%%%%%%%%%%%%%%%%%%%%%%%%%%%%%%%%%%%%%%%%%%%%%%%%%%%%%%%%%%%%%%%%%%%%%%%%%%%%%%%%%%%%%%%%%%%%%%%%%%%%%%
\section{Proofs for the Brownian semistationary process}\label{S8}
\subsection{Strategy and outline of the proof}

In order to prove the central limit theorem for the bivariate Brownian semistationary process we will introduce a \emph{blocking technique}, see \cite{Bernstein1927}, whereby, alongside the original time-grid indexed by $n$, we introduce a coarser grid with a new index $l$, and we freeze the volatility processes at the start of each $l-$interval.
Heuristically, letting $n$ go to infinity, for a fixed $l$, allows us invoke the weak convergence of the Gaussian core we have  proven in the previous section, as the volatilities are ``frozen''. A further limit in $l$ gives us the final result where the volatilities are integrated against the limiting Brownian motion.

Let us now show how the blocking technique will be introduced. 
We define $$\mu_n:=r_{1,2}^{(n)}(0)=\E\left[ \incf n11G \incf n12G\right],$$ which is bounded by 1. For any $ l\leq n$ we have the decomposition:
\begin{align*}
&\frac{1}{\sqrt n} \sum_{i=1}^{\floor {nt}} \incf ni1Y \incf ni2Y-\sqrt n\, \mu_n\int_0^t \sigma\tot1_s\sigma\tot2_s\,ds\\
=&\underbrace{\frac{1}{\sqrt n} \sum_{i=1}^{\floor {nt}}\left(\incf ni1Y \incf ni2Y - \sigma\tot1_{(i-1)\Delta_n}\sigma\tot2_{(i-1)\Delta_n} \incf ni1G \incf ni2G\right)}_{A^n_t}\\
+& \underbrace{\frac{1}{\sqrt n} \sum_{i=1}^{\floor {nt}} \sigma \tot1_{(i-1)\Delta_n}\sigma\tot2_{(i-1)\Delta_n} \incf ni1G\incf ni2G-\frac{1}{\sqrt n}\sum_{j=1}^{\floor{lt}} \sigma\tot1_{(j-1)\Delta l}\sigma \tot 2_{(j-1)\Delta_l} \sum_{i\in I_l(j)} \incf ni1G \incf ni2G}_{A^{'n,l}_t}\\
+&\underbrace{\frac{\sqrt n}{l} \mu_n \sum_{j=1}^{\floor{lt}} \sigma \tot1_{(j-1)\Delta_l}\sigma\tot2_{(j-1)\Delta_l}- \frac {1}{\sqrt n} \mu_n  \sum_{j=1}^{\floor{nt}} \sigma \tot1_{(j-1)\Delta_n}\sigma\tot2_{(j-1)\Delta_n}}_{A^{''n,l}_t}  \\
+& \underbrace{\frac{1}{\sqrt n} \sum_{j=1}^{\floor {lt}} \sigma\tot1_{(j-1)\Delta_l}\sigma\tot2_{(j-1)\Delta_l} \sum_{i\in I_l(j)} \incf ni1G \incf ni2G - \frac{\sqrt n}{l} \mu_n \sum_{j=1}^{\floor{lt}} \sigma\tot1_{(j-1)\Delta_l}\sigma\tot2_{(j-1)\Delta_l}}_{C^{n,l}_t}\\
+&\underbrace{\frac {1}{\sqrt n} \mu_n  \sum_{j=1}^{\floor{nt}} \sigma \tot1_{(j-1)\Delta_n}\sigma\tot2_{(j-1)\Delta_n} -\sqrt n \mu_n \int_0^t \sigma\tot1_s\sigma\tot2_s\,ds}_{D^n_t}.
\end{align*}
The term denoted by  $C^{n,l}_t$ will give us the stable convergence to a non-zero limit, while the terms $A^n_t, A^{'''n,l}_t:=A^{'n,l}_t+A^{''n,l}_t, D^n_t$ will converge to zero (in a way that will be made precise below.)

We will divide the proof into four parts, each one dealing separately with one of the terms.
\subsection{Convergence of the term $A^n_t$}
\begin{proposition}\label{A} Assume that the assumptions of Theorem \ref{CLT} hold. Then 
 $A^n_t$ given by 
\[
A^n_t=\frac{1}{\sqrt n} \sum_{i=1}^{\floor {nt}}\left(\incf ni1Y \incf ni2Y - \sigma\tot1_{(i-1)\Delta_n}\sigma\tot2_{(i-1)\Delta_n} \incf ni1G \incf ni2G\right)
\]
converges to zero  uniformly on compacts in probability (u.c.p.).
\end{proposition}
\begin{proof}[Proof of Proposition \ref{A}]
We write:
\begin{multline}
\incf ni1Y\incf ni2Y\\=\frac{1}{\tau\tot1_n\tau\tot2_n}\left(\DDint g\tot1(i\Delta_n-s)\sigma\tot1_s\,dW\tot1_s+\int_{-\infty}^{(i-1)\Delta_n} \Delta g\tot1\sigma\tot1_s\,dW\tot1_s \right)\\\times\left(\DDint g\tot2(i\Delta_n-s)\sigma\tot2_s\,dW\tot2_s+\int_{-\infty}^{(i-1)\Delta_n} \Delta g\tot2\sigma\tot2_s\,dW\tot2_s \right).
\end{multline}
And the corresponding 4 terms for $\incf ni1G\incf ni2G$.
We start by showing that:
\begin{multline}
\frac{1}{\sqrt n  \tau\tot1_n\tau\tot2_n} \sum_{i=1}^{\floor {nt}}\Biggl[\DDint g\tot1(i\Delta_n-s)\sigma\tot1_s\,dW\tot1_s\DDint g\tot2(i\Delta_n-s)\sigma\tot2_s\,dW\tot2_s-\\\sigma\tot1_{(i-1)\Delta_n}\sigma\tot2_{(i-1)\Delta_n}\DDint g\tot1(i\Delta_n-s)\,dW\tot1_s\DDint g\tot2(i\Delta_n-s)\,dW\tot2_s\Biggr]
\end{multline}
goes to zero.
Adding and subtracting $\sigma\tot1_{(i-1)\Delta_n} \DDint g\tot1(i\Delta_n-s)\,dW\tot1_s\DDint g\tot2(i\Delta_n-s)\sigma\tot2_s\,dW\tot2_s$, we get:
\begin{align}\begin{split}
\label{nkas}
&\frac{1}{\sqrt n \tau\tot1_n\tau\tot2_n} \sum_{i=1}^{\floor {nt}}\DDint g\tot2(i\Delta_n-s)\sigma\tot2_s\,dW\tot2_s\\
&\qquad \qquad\qquad\qquad \times \left[\DDint g\tot1(i\Delta_n-s)\left(\sigma\tot1_s-\sigma\tot1_{(i-1)\Delta_n}\right)\,dW\tot1_s\right]\\
&+
\frac{1}{\sqrt n \tau\tot1_n\tau\tot2_n} \sum_{i=1}^{\floor {nt}} \sigma\tot1_{(i-1)\Delta_n}\DDint g\tot1(i\Delta_n-s)\,dW\tot1_s\\
& \qquad\qquad\qquad\qquad \times \left[\DDint g\tot2(i\Delta_n-s)\left(\sigma\tot2_s-\sigma\tot2_{(i-1)\Delta_n}\right)\,dW\tot2_s\right].
\end{split}\end{align}
We can show  $u.c.p$  convergence to zero. If we can show that the supremum over $[0,T]$ converges in $L^1$  to zero, it would be enough.
If we take the first term of \eqref{nkas},
\begin{align*}
&\E\left|\frac{1}{\sqrt n \tau\tot1_n\tau\tot2_n} \sum_{i=1}^{\floor {nt}}\DDint g\tot2(i\Delta_n-s)\sigma\tot2_s\,dW\tot2_s \right.
\\ &\qquad\qquad\qquad\qquad\left. \times \left[\DDint g\tot1(i\Delta_n-s)\left(\sigma\tot1_s-\sigma\tot1_{(i-1)\Delta_n}\right)\,dW\tot1_s\right]\right|
\end{align*}
that is smaller than:
\begin{align*}
&
\frac{1}{\sqrt n \tau\tot1_n\tau\tot2_n} \sum_{i=1}^{\floor {nt}}\E\left|\DDint g\tot2(i\Delta_n-s)\sigma\tot2_s\,dW\tot2_s
\right.
\\ &\qquad\qquad\qquad\qquad\left. \times
\left[\DDint g\tot1(i\Delta_n-s)\left(\sigma\tot1_s-\sigma\tot1_{(i-1)\Delta_n}\right)\,dW\tot1_s\right]\right|.
\end{align*}
By Cauchy-Schwarz $\E[|XY|]\leq \sqrt {\E[X^2]}\sqrt {\E[Y^2]}$.  Now:
\[
\sqrt{\DDint \left(g\tot2(i\Delta_n-s)\right)^2\E[\sigma\tot2_s]\,ds}=\sqrt{\int_0^{\Delta_n}\left(g\tot2(s)\right)^2\E\left[\left(\sigma\tot2_{i\Delta_n-s}\right)^2\right]\,ds},
\]
since $\sigma$ is bounded on compact intervals, we get the bound
\[
K\frac{\sqrt{\int_0^{\Delta_n}\left(g\tot2(s)\right)^2\,ds}}{\tau\tot2_n} \times \frac{1}{\sqrt n \tau\tot1_n} \sum_{i=1}^{\floor {nt}}\sqrt{\E\left[\left(\DDint g\tot1(i\Delta_n-s)(\sigma\tot1_s-\sigma\tot1_{(i-1)\Delta_n-s})\,dW\tot1_s\right)^2\right]},
\]
for some constant $K$.
Now the first term is bounded by $K$, and the second one is term $A$ in the proof of Theorem 5 in \cite{barndorff2011multipower} (page 37, full version), which goes to zero, under our assumptions.
We can repeat the reasoning for the second term of \eqref{nkas}. 
Let's take another term now:
\begin{multline}
\frac{1}{\sqrt n\tau\tot1_n\tau\tot2_n} \sum_{i=1}^{\floor {nt}} \Biggl[\int_{-\infty}^{(i-1)\Delta_n}  \Delta g\tot1 \sigma\tot1_s\,dW\tot1_s \int_{-\infty}^{(i-1)\Delta_n} \Delta g\tot2 \sigma\tot2_s\,dW\tot2_s\\- \sigma\tot1_{(i-1)\Delta_n}\sigma\tot2_{(i-1)\Delta_n} \int_{-\infty}^{(i-1)\Delta_n} \Delta g\tot1 \,dW\tot1_s \int_{-\infty}^{(i-1)\Delta_n} \Delta g\tot2 \,dW\tot2_s\Biggr].
\end{multline}
Adding and subtracting $\sigma\tot1_{(i-1)\Delta_n}\int_{-\infty}^{(i-1)\Delta_n}\Delta g\tot1\,dW\tot1_s \int_{-\infty}^{(i-1)\Delta_n} \Delta g\tot2 \sigma\tot2_s\,dW\tot2_s$, we get as the first term:
\[
\frac{1}{\sqrt n} \sum_{i=1}^{\floor {nt}}  \underbrace{\frac{\int_{-\infty}^{(i-1)\Delta_n} \Delta g\tot2 \sigma\tot2_s\,dW\tot2_s}{\tau\tot2_n}}_{(1)} \underbrace{\frac{ \left[ \int_{-\infty}^{(i-1)\Delta_n}  \Delta g\tot1 \sigma\tot1_s\,dW\tot1_s-\sigma\tot1_{(i-1)\Delta_n}\int_{-\infty}^{(i-1)\Delta_n}\Delta g\tot1\,dW\tot1_s\right]}{\tau\tot1_n}}_{(2)}
\]
We can use the same arguments as above. The only difference is the expectation of (1) over the infinite interval:
\[
\frac{\sqrt{\E\left[\left(\int_{-\infty}^{(i-1)\Delta_n} \Delta g\tot2 \sigma\tot2_s\,dW\tot2_s\right)^2\right]}}{\tau\tot2_n}=\frac{\sqrt{\int_0^\infty \left(g\tot2(s+\Delta_n)-g\tot2(s)\right)^2\E\left[\left(\sigma_{(i-1)\Delta_n-s}\tot2\right)^2\right]\,ds}}{\tau\tot2_n}
\]
Assumption \ref{uetsygvs} allows to conclude that this quantity is bounded. The remaining term (2) is the sum  $B+C$  from the paper  \cite{barndorff2011multipower} and goes to zero in $L^2$ as well.

Now we consider the cross term
\begin{multline}
\frac{1}{\sqrt n} \sum_{i=1}^{\floor{nt}} \Biggl(\DDint g\tot1(i\Delta_n-s)\sigma\tot1_s\,dW\tot1_s \int_{-\infty}^{(i-1)\Delta_n} \Delta g\tot2 \sigma\tot2_s\,dW\tot2_s  \\ -\sigma\tot1_{(i-1)\Delta_n}\sigma\tot2_{(i-1)\Delta_n}   \DDint g\tot1(i\Delta_n-s)\,dW\tot1_s \int_{-\infty}^{(i-1)\Delta_n} \Delta g\tot2 \,dW\tot2_s  \Biggr).
\end{multline}
We add and subtract: $\sigma\tot1_{(i-1)\Delta_n}\DDint g\tot1(i\Delta_n-s)\,dW\tot1_s\int_{-\infty}^t\Delta g\tot2 \sigma\tot2_s\,dW\tot2_s$.
\begin{multline}
\frac{1}{\sqrt n} \sum_{i=1}^{\floor{nt}} \int_{-\infty}^{(i-1)\Delta_n} \Delta g\tot2\sigma\tot2_s\,dW\tot2_s \left(\DDint g\tot1(i\Delta_n-s)(\sigma\tot1_s-\sigma\tot1_{(i-1)\Delta_n})\,dW\tot1_s\right)\\
+\frac{1}{\sqrt n} \sum_{i=1}^{\floor{nt}} \sigma\tot1_{(i-1)\Delta_n}\DDint g\tot1(i\Delta_n-s)\,dW\tot1_s\left( \int_{-\infty}^{(i-1)\Delta_n}  \Delta g\tot2 \left(\sigma\tot2_s - \sigma\tot2_{(i-1)\Delta_n}\right)\,dW\tot2_s\right).
\end{multline}
We can proceed exactly as above, and convergence to zero is proved.
\end{proof}

%%%%%%%%%%%%%%%%%%%%%%%%%%%%%
\subsection{Convergence of the term $A^{'''n,l}_t=A^{'n,l}_t+A^{''n,l}_t$}
It is worth mentioning at this point that proofs that terms similar to the one we called $A^{'''n,l}_t$ converge to zero in the univariate case have had a tormented history in the literature. Indeed, a mistake appeared in the proof of a similar result in \cite{corcuera2006power} in the context of power variation for integral processes. The application of the mean value theorem on page   724  of that paper is invalid. 

The mistake was not simple to correct. Years later, the  paper \cite{corcuera2014asymptotics} was published, which highlighted the techniques from fractional integration that were needed to correct the proof.
As it turns out, in our multivariate setting it is sufficient to invoke that univariate result to obtain the required convergence. This section contains the details of the proof.

\begin{proposition}\label{A'''}
Assume that the assumptions of Theorem \ref{CLT} hold. Then 
\[
\Prob- \lim_{l\to \infty} \limsup_{n\to\infty} \sup_{t\in[0,T]} \left|A^{'''n,l}_t \right| = 0.
\]
\end{proposition}
\begin{proof}[Proof of Proposition \ref{A'''}]
We need to set the following notation:
 $$\xi_{i,m}=\frac{1}{\sqrt m}\left(\incf mi1G \incf m12G-\E\left[\incf mi1G \incf m12G\right]\right), $$
and $f(t_i)=\sigma\tot1_{(i-1)\Delta_n}\sigma\tot2_{(i-1)\Delta_n}$. We will be using Remark 1.1 in the paper \cite{corcuera2014asymptotics}.
We know that: $$\sum_{i=1}^{\floor{mt}} \xi_{i,m} \Rightarrow \sqrt \beta W_t.$$
Convergence (4) in the paper reads:
\[
\Prob- \lim_{n\to \infty} \limsup_{m\to\infty} \sup_{t\in[0,T]} \left| \sum_{j=1}^{\floor{nt}}\sum_{i\in I_n(j)}\left(f(t_i)-f(u_{j-1})\right)\xi_{i,m}\right|=0,
\]
which in our setting and with our notation becomes:
\begin{multline*}
\Prob- \lim_{l\to \infty} \limsup_{n\to\infty} \sup_{t\in[0,T]}\Biggl|\sum_{j=1}^{\floor{lt}+1} \sum_{i\in I_l(j)} \left( \sigma\tot 1_{(i-1)\Delta_n}\sigma\tot 2_{(i-1)\Delta_n}-\sigma\tot 1_{(j-1)\Delta_l}\sigma\tot 1_{(j-1)\Delta_l}\right)\\\times\frac{1}{\sqrt n}\underbrace{\left(\incf ni1G\incf ni2G-\mu_n\right)}_{(1)}\Biggr|=0.
\end{multline*}
Expanding the bracket in (1) above, the first term gives us exactly term $A_t^{'n,l}$. The second term from the bracket (1) is:
\[
\begin{split}
&\Biggl|{\frac{\mu_n}{\sqrt n}\sum_{j=1}^{\floor{lt}+1} \sum_{i\in I_l(j)} \sigma \tot1_{(j-1)\Delta_n}\sigma \tot2_{(j-1)\Delta_n} -\frac{\mu_n}{\sqrt n}\sum_{j=1}^{\floor{lt}+1}\sum_{i\in I_l(j)}  \sigma \tot1_{(i-1)\Delta_n}\sigma \tot2_{(i-1)\Delta_n}}\Biggr|\\
\leq& \Biggl|\frac{\mu_n}{\sqrt n}\max_j\{\#I_l(j)\}\sum_{j=1}^{\floor{lt}+1}\sigma \tot1_{(j-1)\Delta_n}\sigma \tot2_{(j-1)\Delta_n}  -\frac{\mu_n}{\sqrt n}\sum_{j=1}^{\floor{lt}+1}\sum_{i\in I_l(j)}  \sigma \tot1_{(i-1)\Delta_n}\sigma \tot2_{(i-1)\Delta_n}\Biggr|\\
\leq & \Biggl|\frac {\mu_n}{\sqrt n} \left(\frac nl+1\right)\sum_{j=1}^{\floor{lt}+1}\sigma \tot1_{(j-1)\Delta_n}\sigma \tot2_{(j-1)\Delta_n}  -\frac{\mu_n}{\sqrt n}\sum_{j=1}^{\floor{lt}+1}\sum_{i\in I_l(j)}  \sigma \tot1_{(i-1)\Delta_n}\sigma \tot2_{(i-1)\Delta_n} \Biggr|
\end{split}
\]
which is bounded by $ \Biggl|A_t^{''n,l}+\frac {\mu_n}{\sqrt n}\sum_{j=1}^{\floor{lt}+1}\sigma \tot1_{(j-1)\Delta_n}\sigma \tot2_{(j-1)\Delta_n} \Biggr|$. This second term goes to zero a.s. for any fixed $l$, which concludes the proof. 

\end{proof}
\subsection{Convergence of the term $C^{n,l}_t$}
The term $C^{n,l}_t$ is the one that will give us the stable convergence we seek.
\begin{proposition}\label{C1}
Assume that the assumptions of Theorem \ref{CLT} hold. Then 
\[
\begin{pmatrix} G\tot1_t ,G\tot2_t,\frac{1}{\sqrt n} \sum_{i=1}^{\floor {nt}} \left(\incf ni1G \incf ni2G - \mu_n \right) \end{pmatrix}_{t\in[0,T]}
\]
converges weakly to 
\[
\begin{pmatrix} G\tot1_t ,G\tot2_t,\sqrt \beta B_t \end{pmatrix}_{t\in[0,T]}.
\]
\end{proposition}
\begin{proof}[Proof of Proposition \ref{C1}]
We split the proof into two parts: First, we prove  tightness and then convergence of the finite dimensional distributions. 

\noindent{\bf Tightness:}
By Theorem 13.2 in \cite{billingsley2009convergence}, a sequence of measures $\Prob_n$ is tight if and only if:
\begin{enumerate}
\item $\lim_{a\to\infty} \limsup_{n\to\infty} \Prob_n\left(x| \norm x \geq a\right)=0$
\item For any $\epsilon>0$, $\lim_{\delta} \limsup_{n\to\infty} \Prob_n\left(x| w'_x(\delta)\geq \epsilon\right)=0$,
\end{enumerate}
where $w'_x(\delta)$ is defined as follows. For $S\subseteq [0,T]$, call the \emph{modulus of continuity} $$w_x(S):= \sup_{s,t\in S} \norm{f(s)-f(t)}_{\R^3}.$$ A partition $0\leq t_1\leq\dots\leq t_v=T$ of $[0,T]$ is $\delta$-\emph{sparse}  if $\min_{1\leq i\leq v} (t_i-t_{i-1})>\delta$. Define, for $0<\delta<T$, $$w'_x(\delta)=\inf_{t_i} \max_{1\leq i\leq v}w_x[t_{i-1},t_i),$$ where the maximum runs over all $\delta$-sparse partitions $\{t_i\}$.

Our probability measures $\Prob_n$ live in $\mathcal D \left([0,T];\R^3\right),$ the space of \cadlag functions with values in $\R^3$, equipped with the Skorokhod topology. The norm in this space is defined as:
\[
\norm{f}_{\mathcal D \left([0,T];\R^3\right)}=\sup_{t\in[0,T]} \norm{f}_{\R^3},
\]
and hence the two conditions above only depend on the norm in $\R^3$. It is then sufficient to show them component-wise. The first two components trivially satisfy them, as the sequences reduce to only one measure per component.
The fact that the third component satisfies them both is a consequence of Theorem \ref{tightnesstheo} and the characterisation above.

\noindent{\bf Convergence of the finite dimensional distributions:}
We need to show that for any choice of positive numbers $a_k<b_k, k\in\{1,\dots,D\}$, the sequence of matrix variables:
\[
\begin{pmatrix} G\tot1_{b_k} -G\tot1_{a_k},G\tot2_{b_k}-G\tot2_{a_k},\frac{1}{\sqrt n}\sum_{i=\floor{n a_k}+1}^{\floor {nb_k}} \left(\incf ni1G \incf ni2G - \mu_n \right)\end{pmatrix}_{1\leq k\leq D}
\]
converges in law, as $n\to \infty$, to:
\begin{equation}
\label{poloopl}
\begin{pmatrix} G\tot1_{b_k} -G\tot1_{a_k},G\tot2_{b_k}-G\tot2_{a_k},\sqrt \beta \left(B_{b_k}-B_{a_k}\right)\end{pmatrix}_{1\leq k\leq D}.
\end{equation}
This we know already, as pointed out in the proof of Theorem \ref{weakconv}, as marginal convergence of sequence of variables within fixed Wiener chaoses implies joint convergence. The first two components lie in the first chaos, the third one lies in the second chaos.
The statement of Theorem \ref{weakconv} allows to conclude.
\end{proof}
\begin{proposition}\label{C2}
Assume that the assumptions of Theorem \ref{CLT} hold. Then $C_t^n$ converges stably in law to
$ \sqrt \beta \int_0^t \sigma_s \tot1 \sigma_s\tot 2\,dB_s$ in the Skorokhod space $ \mathcal D[0,T]$, where $\beta={ C(1,1)}$, see equation  \eqref{mknmknm}, and where first $n\to \infty$ for fixed $l$ and then $l\to \infty$.
Also, $B$ is Brownian motion, independent of $\mathscr F$ and defined on  an extension of the filtered probability space $(\Omega,\mathscr F,\mathscr F_t,\Prob)$. 
\end{proposition}
\begin{proof}[Proof of Proposition \ref{C2}]
The joint weak convergence in \eqref{poloopl} paired with the asymptotic independence of the limit $B$ and $G\tot1,G\tot2$ and an application of  Proposition \ref{dsknskrnavrkn} ensure that:
\[
\frac{1}{\sqrt n} \sum_{i=1}^{\floor {nt}} \left(\incf ni1G \incf ni2G - \mu_n \right)\Rightarrow \sqrt \beta B_t\qquad \text(mixing).
\]

Applying the continuous mapping Theorem \ref{continuousmapping} with the  sigma-algebra $\mathscr G$,  $\sigma\tot1_{(j-1)\Delta_l}\sigma\tot2_{(j-1)\Delta_l}$ as the measurable variable $\sigma$, $\frac{1}{\sqrt n}\sum_{i\in I_l(j)}\left(\incf ni1G\incf ni2G -\mu_n\right)$ as   $Y_n$
  and $g(x,y)=xy$,  since $Y_n\overset{st.}{\Rightarrow} \sqrt \beta \left(B_{j\Delta_l}-B_{(j-1)\Delta_l}\right)$, we have the following $\mathscr G$-stable convergence for fixed $l$ as $n\to \infty$:
\[
\sigma\tot1_{(j-1)\Delta_l}\sigma\tot2_{(j-1)\Delta_l}\frac{1}{\sqrt n}\sum_{i\in I_l(j)}\left(\incf ni1G\incf ni2G -\mu_n\right)\stable \sigma\tot1_{(j-1)\Delta_l}\sigma\tot2_{(j-1)\Delta_l}  \sqrt \beta \left(B_{j\Delta_l}-B_{(j-1)\Delta_l}\right).
\]
Finally we have that
\[
\Prob-\lim_{l\to \infty} \sum_{j=1}^{\floor{lt}}  \sigma\tot1_{(j-1)\Delta_l}\sigma\tot2_{(j-1)\Delta_l}  \sqrt \beta \left(B_{j\Delta_l}-B_{(j-1)\Delta_l}\right)=\sqrt\beta\int_0^t \sigma\tot1_s\sigma\tot2_s\,dB_s,
\]
because the integrand is c\`adl\`ag.
Modulo another term of the form $\frac {\mu_n}{\sqrt n}\sum_{j=1}^{\floor{lt}+1}\sigma \tot1_{(j-1)\Delta_n}\sigma \tot2_{(j-1)\Delta_n}$, which goes to zero a.s. as $n\to\infty$, we have proven stable convergence of the term $C_t^{n,l}$ in our decomposition.
\end{proof}
\subsection{Convergence of the term $D^n_t$}
\begin{proposition}\label{D}
Assume that the assumptions of Theorem \ref{CLT} hold. Then $\sup_{t\in[0,T]} \abs {D^n_t}\to 0$ almost surely.
\end{proposition}
\begin{proof}[Proof of Proposition \ref{D}]
Note that  $D^n_t$ is given by 
\[
\frac {1}{\sqrt n} \mu_n  \sum_{j=1}^{\floor{nt}} \sigma \tot1_{(j-1)\Delta_n}\sigma\tot2_{(j-1)\Delta_n} -\sqrt n \mu_n \int_0^t \sigma\tot1_s\sigma\tot2_s\,ds.
\]
Recall that $\alpha\tot i$ denotes the H\"older continuity index of $\sigma\tot i$.
Rewriting the integral:
\[
\int_0^t \sigma\tot1_s\sigma\tot2_s\,ds=\sum_{j=1}^{\floor{nt}} \int_{(j-1)\Delta_n}^{j\Delta_n} \sigma\tot1_s\sigma\tot2_s\,ds + \int_{\floor{nt}\Delta_n}^t \sigma\tot1_s\sigma\tot2_s\,ds,
\]
and using the mean value theorem, we get:
\[
\begin{split}
\abs{D^n_t} \leq &\frac{1}{\sqrt n} \mu_n \left( \sum_{j=1}^{\floor {nt}} \abs{\sigma\tot 1_{(j-1)\Delta_n}\sigma\tot2_{(j-1)\Delta_n}-\sigma\tot1_{s_j}\sigma\tot2_{s_j}}\right)+ \frac {1}{\sqrt n} \mu_n \norm{\sigma\tot1_{s_j}\sigma\tot2_{s_j}}_{\infty}\\
\leq&\frac{1}{\sqrt n} \mu_n \left( \sum_{j=1}^{\floor {nt}}\abs{(j-1)\Delta_n-s_j}^{\min(\alpha\tot1,\alpha\tot2)} \abs{\sigma\tot1_{(j-1)\Delta_n}+\sigma\tot2_{s_j}}\right)+ \frac {1}{\sqrt n} \mu_n \norm{\sigma\tot1_{s_j}\sigma\tot2_{s_j}}_{\infty}\\
\leq& C \frac{1}{\sqrt n} \mu_n {\Delta_n}^{\min(\alpha\tot1,\alpha\tot2)} nT l + \frac{1}{\sqrt n}\mu_n \norm{\sigma\tot1_{s_j}\sigma\tot2_{s_j}}_{\infty}\\=&C \sqrt n \mu_n {\Delta_n}^{\min(\alpha\tot1,\alpha\tot2)} T  + \frac{1}{\sqrt n}\mu_n \norm{\sigma\tot1_{s_j}\sigma\tot2_{s_j}}_{\infty}.
\end{split}
\]
Hence, $\sup_{t\in[0,T]} \abs {D^n_t}\to 0$ almost surely, since ${\min(\alpha\tot1,\alpha\tot2)} > \frac12$.
\end{proof}

\subsection{Proofs of Theorem \ref{CLT} and Proposition \ref{LLN}}
\begin{proof}[Proof of Theorem \ref{CLT}]
The statement of Theorem \ref{CLT} is a consequence of Propositions \ref{A}, \ref{A'''}, \ref{C2}, \ref{D}, noting that they imply that, for any $\epsilon >0$,
\[
\lim_{l\to \infty}\limsup_{n\to\infty}\Prob \left(\sup_{t\in[0,T]} \left| A^n_t + A^{' n,l}_t + A^{''n,l}_t + D^n_t\right| \geq \epsilon \right)=0.
\]
It is now sufficient to apply  Theorem 3.2 in \cite{billingsley2009convergence} to conclude.
 \end{proof}

Finally we provide the proof of the weak law of large numbers.
\begin{proof}[Proof of Proposition \ref{LLN}]
We note that, for each fixed $t\in [0,T]$, \eqref{CLTFormula} implies that:
\[
\left\{\sqrt n\left(\frac{1}{ n} \sum_{i=1}^{\floor {nt}} \incf ni1Y \incf ni2Y- \E\left[ \incf n11G \incf n12G\right]\int_0^t \sigma\tot1_s\sigma\tot2_s\,ds\right)\right\}_{n\in\N}
\]
converges weakly, hence, by Prohorov's theorem, it is a tight sequence.
It then follows that:
\[
\frac{1}{ n} \sum_{i=1}^{\floor {nt}} \incf ni1Y \incf ni2Y- \E\left[ \incf n11G \incf n12G\right]\int_0^t \sigma\tot1_s\sigma\tot2_s\,ds \overset{\Prob}{\to}  0.
\]
Now:
\begin{align*}
&\E\left[ \incf n11G \incf n12G\right]\\&=\frac{\int_0^{\Delta_n} g\tot1 (s) g\tot2 (s)\rho\,ds + \int_0^\infty \left(g\tot 1(s+\Delta_n)-g\tot1 (s)\right)\left(g\tot 2(s+\Delta_n)-g\tot2 (s)\right)\rho\,ds}{\tau\tot1 _n\tau\tot2_n }\\&=\rho \frac{c(\Delta_n)}{\tau\tot1_n\tau\tot2_n}.
\end{align*}
Hence:
\[
{\Delta_n} \sum_{i=1}^{\floor {nt}} \incf ni1Y \incf ni2Y   -    \rho\frac{c(\Delta_n)}{\tau\tot1_n\tau\tot2_n} \int_0^t \sigma\tot1_s\sigma\tot2_s\,ds \overset{\Prob}{\to}0,
\]
which is equivalent to:
\[
\Delta_n\sum_{i=1}^{\floor{nt}} \Delta^n_iY\tot1 \Delta_i^nY\tot2 -\rho c(\Delta_n) \int_0^t \sigma\tot1_s\sigma\tot2_s\,ds \overset{\Prob}{\to}0,
\]
or indeed to:
\[
\frac{\Delta_n}{c(\Delta_n)} \sum_{i=1}^{\floor{nt}} \Delta^n_iY\tot1 \Delta_i^nY\tot2 \overset{\Prob}{\to} \rho\int_0^t \sigma\tot1_s\sigma\tot2_s\,ds.
\]
\end{proof}

\section*{Acknowledgement}
We wish to thank Damiano Brigo, Dan Crisan, Mikko Pakkanen,  Mark Podolskij and Riccardo Passeggeri   for helpful discussions. 
AG is grateful to the Department of Mathematics of Imperial College for his PhD scholarship which supported this research. AEDV acknowledges
financial support by a Marie Curie FP7 Integration Grant within the 7th European Union Framework Programme (grant agreement number PCIG11-GA-2012-321707).

%\printbibliography

% \bibliographystyle{agsm}
% \bibliography{POST_PHDbibliographyUPDATED}

\end{document}